\documentclass{amsart}
\usepackage{bm}
\usepackage{amsmath,amssymb, amscd}
\usepackage{graphicx}

\newtheorem{theorem}{Theorem}[section]
\newtheorem{lemma}[theorem]{Lemma}
\newtheorem{proposition}[theorem]{Proposition}
\newtheorem{corollary}[theorem]{Corollary}

\theoremstyle{definition}
\newtheorem{definition}[theorem]{Definition}
\newtheorem{example}[theorem]{Example}
\numberwithin{equation}{section}

\newtheorem{remark}[theorem]{Remark}

\begin{document}
\date{\today}
\title{ 
Local B-model and Mixed Hodge Structure
}
\author{Yukiko Konishi}
\address{ 
Department of Mathematics, 
Kyoto University,
Kyoto 606-8502 JAPAN 
}
\email{konishi@math.kyoto-u.ac.jp}
\author{Satoshi Minabe}
\address{
Max-Planck-Instutut f\"ur Mathematik,
Vivatsgasse 7, 53111 Bonn, Germany}
\email{minabe@mpim-bonn.mpg.de}

\begin{abstract}
We study the mixed Hodge theoretic aspects of the 
B-model side of local mirror symmetry.  
Our main objectives are to define an analogue of the Yukawa 
coupling in terms of the variations 
of the mixed Hodge structures and to study its properties.  
We also describe a local version of 
Bershadsky--Cecotti--Ooguri--Vafa's
holomorphic anomaly equation.
\end{abstract}
\maketitle

\newcommand{\V}{\mathbf{V}}
\newcommand{\T}{\mathbb{T}}
\newcommand{\polytope}{\tilde{\Delta}}
\newcommand{\surface}{S_{\Delta}} 
\newcommand{\ring}{\mathbf{S}}
\newcommand{\divisor}{\mathbb{D}}
\newcommand{\res}{\mathrm{Res}}
\newcommand{\hodge}{\mathcal{F}}
\newcommand{\weight}{\mathcal{W}}
\newcommand{\threeform}{\omega}
\newcommand{\moduli}{\mathcal{M}(\Delta)}
\newcommand{\GM}{\nabla}
\newcommand{\koszul}{\mathcal{K}}
\newcommand{\D}{\mathcal{D}'}
\newcommand{\M}{\mathbf{M}}

\newcommand{\hyper}{\mathcal{T}}
\newcommand{\wronskian}{\mathrm{Wr}}
\newcommand{\KS}{\varrho}
\newcommand{\contraction}{\iota}
\section{Introduction}
\subsection{Local mirror symmetry}
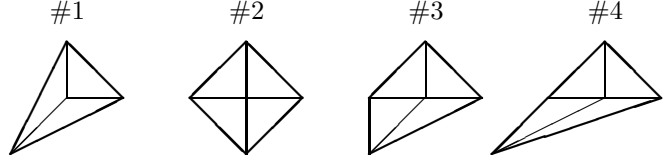
\begin{figure}[t]
\unitlength .075cm

\begin{picture}(22,30)(-11,-10) 
\thinlines
\put(0,0){\line(1,0){10}}
\put(0,0){\line(0,1){10}}
\put(0,0){\line(-1,-1){10}}
\thicklines
\put(10,0){\line(-1,1){10}}
\put(0,10){\line(-1,-2){10}}
\put(-10,-10){\line(2,1){20}}
\put(-3,14){$\#$1}
\end{picture}
\hspace*{.5cm}
\begin{picture}(22,30)(-11,-10) 
\thinlines
\put(0,0){\line(1,0){10}}
\put(0,0){\line(0,1){10}}
\put(0,0){\line(-1,0){10}}
\put(0,0){\line(0,-1){10}}
\thicklines
\put(10,0){\line(-1,1){10}}
\put(0,10){\line(-1,-1){10}}
\put(-10,0){\line(1,-1){10}}
\put(0,-10){\line(1,1){10}}
\put(-3,14){$\#$2}
\end{picture}
\hspace*{.5cm}
\begin{picture}(22,30)(-11,-10) 
\thinlines
\put(0,0){\line(1,0){10}}
\put(0,0){\line(0,1){10}}
\put(0,0){\line(-1,0){10}}
\put(0,0){\line(-1,-1){10}}
\thicklines
\put(10,0){\line(-1,1){10}}
\put(0,10){\line(-1,-1){10}}
\put(-10,0){\line(0,-1){10}}
\put(-10,-10){\line(2,1){20}}
\put(-3,14){$\#$3}
\end{picture}
\hspace*{.5cm}
\begin{picture}(22,30)(-11,-10)
\thinlines
\put(0,0){\line(1,0){10}}
\put(0,0){\line(0,1){10}}
\put(0,0){\line(-1,0){10}}
\put(0,0){\line(-2,-1){20}}
\thicklines
\put(10,0){\line(-1,1){10}}
\put(0,10){\line(-1,-1){10}}
\put(-10,0){\line(-1,-1){10}}
\put(-20,-10){\line(3,1){30}}
\put(-3,14){$\#$4}
\end{picture}
\caption{Examples of $2$-dimensional reflexive polyhedra. 
($\mathbb{P}^2$,
$\mathbb{F}_0,\mathbb{F}_1,\mathbb{F}_2$ cases.)}\label{fig:Hirzebruch}
\end{figure}
Mirror symmetry 
states a relationship between 
the genus zero Gromov--Witten theory (``A-model'') of a 
Calabi--Yau threefold $X$ and 
the Hodge theory (``B-model'') 
of its mirror Calabi--Yau threefold $X^{\vee}$.
After the first example of a quintic hypersurface in $\mathbb{P}^4$
and its mirror \cite{GreenePlesser, Candelas},
Batyrev \cite{Batyrev-mirror} showed that a mirror pair of Calabi--Yau hypersurfaces
in toric varieties can be 
constructed from a reflexive polyhedron%
\footnote{See \S \ref{sec:reflexive} for the definition of reflexive polyhedra.
}.
Local mirror symmetry was 
derived from mirror symmetry for toric Calabi--Yau hypersurfaces
by considering a certain limit in the K\"ahler and complex moduli spaces%
\footnote{
This limit typically corresponds to a situation 
on the A-model side
where one considers the effect of
a {\it local} geometry of a weak Fano surface 
within a Calabi--Yau threefold. 
Hence the term ``local mirror symmetry''. 
See \cite[\S 4]{KatzKlemmVafa}.}
\cite{KatzKlemmVafa} \cite{CKYZ}.
Chiang--Klemm--Yau--Zaslow \cite{CKYZ}  gave quite a thorough mathematical 
treatment to it.
Their result can be summarized as follows.

Take a  $2$-dimensional reflexive polyhedron $\Delta$
(see Figure \ref{fig:Hirzebruch} for examples). 
On one side (``lcoal A-model'' side), 
one considers the genus zero local Gromov--Witten (GW) invariants
of a   
smooth weak Fano toric surface $\mathbb{P}_{\Sigma(\Delta^*)}$
which is determined by the $2$-dimensional 
complete fan $\Sigma(\Delta^*)$ generated by integral points
of $\Delta$. 
On the other side (``local B-model'' side),
one considers  a system of differential equations 
associated to $\Delta$ 
called the  $A$-hypergeometric system
with parameter zero
due to
Gel'fand--Kapranov--Zelevinsky
\cite{GKZ1,GKZ2}.
Then the statement of local mirror symmetry is that 
the genus zero local GW invariants can be obtained
from solutions of the $A$-hypergeometric system.
\begin{remark}
The problem of computing the local GW invariants, 
not only at genus zero but also at all genera, 
is solved completely
by the method of the topological vertex \cite{AKMV}.
\end{remark}

\subsection{Local B-model and the mixed Hodge structure}\label{sec:intro2}
When one compares local mirror symmetry with 
mirror symmetry,
it is easy to see an analogy between the A-model (GW invariants)
and the local A-model (local GW invariants).
To compare the B-model and the local B-model,
let us look into them in more detail.
A natural framework for the B-model is 
the variation of polarized Hodge structures on $H^3(X^{\vee})$
\footnote{Throughout the paper, 
the coefficient of the cohomology group is $\mathbb{C}$
unless otherwise specified.
}
(cf. \cite[Ch.5]{CoxKatz}, \cite[Ch.1, Ch.3]{Voisin-mirror}).
One considers 
\begin{enumerate}
\item[(i)] the family $\pi:\mathcal{X}\to B$
of complex deformations
of the Calabi--Yau threefold $X^{\vee}$, 
\item[(ii)] 
a relative holomorphic three form $\Omega_{\mathcal{X}/B}$
which, together with the elements obtained by successive applications of
the Gauss--Manin connection $\nabla$, 
spans $H^3(X^{\vee})$,
\item[(iii)]
the Picard--Fuchs system for period integrals of $\Omega_{\mathcal{X}/B}$,
\item[(iv)] 
an $\mathcal{O}_B$-multilinear symmetric map from 
$TB\times TB\times TB$ to $\mathcal{O}_B$
called
the 
Yukawa coupling:
$$
Y_b(A_1,A_2,A_3)=\int_{X^{\vee}_b}\nabla_{A_1}\nabla_{A_2}\nabla_{A_3} 
\Omega_{\mathcal{X}/B}\wedge \Omega_{\mathcal{X}/B}\, , 
\qquad(b\in B)~.
$$
\end{enumerate}

Let us turn to the local B-model. Our proposal in this paper is that 
a natural language for the local B-model is 
the mixed Hodge structures and their variations.
The mixed Hodge structure (MHS) 
due to Deligne \cite{Deligne2} is
a generalization of the Hodge structure with the extra data
$\mathcal{W}_{\bullet}$ called the weight filtration. 
See \S \ref{section:Preliminaries}. 
Although the cohomology $H^*(V^{\circ})$
of an open smooth variety $V^{\circ}$ does not have a Hodge structure
in general, it does have a canonical MHS \cite{Deligne2, Deligne3}. 
There is also a canonical  
one on the relative cohomology 
$H^*(U^{\circ},V^{\circ})$.

Now, let us explain what are the counterparts of (i)--(iv) in the local B-model. 
Let $\Delta$ be a 2-dimensional reflexive polyhedron
as above and  $F_a$ be 
a $\Delta$-regular Laurent polynomial, i.e., 
a Laurent polynomial of the form
$$
F_a(t_1,t_2)=\sum_{m\in \Delta\cap \mathbb{Z}^2} a_m t^m~\, \in \mathbb{C}
[t_1^{\pm 1},t_2^{\pm 1}]
$$
which satisfies a certain regularity condition (cf. Definition \ref{def:regular}). 
In the literature, two closely related manifolds
associated to $F_a$ are considered: 
the one is the affine curve  $C^{\circ}_a$ 
in the $2$-dimensional algebraic torus $\mathbb{T}^2=(\mathbb{C}^*)^2$
defined by $F_a(t_1,t_2)=0$ \cite[\S 6]{CKYZ}, 
and the other is the open threefold 
$Z^{\circ}_a\subset \mathbb{T}^2\times \mathbb{C}^2$ 
defined by $F_a+xy=0$ \cite[\S 8]{HIV}.
As we shall see, they give the same result. 
By varying the parameter $a=(a_m)$, we have 
a family of affine curves 
$\mathcal{Z}\rightarrow \mathbb{L}_{{\rm reg}}(\Delta)$
and a family of open threefolds $\mathcal{Z}'\rightarrow 
\mathbb{L}_{{\rm reg}}(\Delta)$.
By taking a quotient by the following action of $\T^3=(\mathbb{C}^*)^3$, 
$$F_a(t_1,t_2)\mapsto \lambda_0 F_a(\lambda_1 t_1,\lambda_2 t_2)\, , \quad 
(\lambda_0,\lambda_1,\lambda_2)\in \T^3, 
$$
we also have
the quotient families $\mathcal{Z}/\T^3\to \moduli$
and $\mathcal{Z}'/\T^3\rightarrow \moduli$. 
These correspond to (i).
As a counterpart of (ii),
we consider, for the affine curve $C^{\circ}_a$, the class
$$
\omega_0
=\Big[\Big(
\frac{dt_1}{t_1}\wedge \frac{dt_2}{t_2}\,,\,0
\Big)\Big] 
 \in H^2(\T^2,C^{\circ}_a)~
\footnote{Note that the class $\omega_0$
depends on the parameter $a$, 
although it is not indicated in the notation.}
~,
$$
in the relative cohomology $H^2(\T^2,C^{\circ}_a)$, 
and for the open threefold $Z^{\circ}_a$,
the class of a holomorphic $3$-form:
$$
\omega_a=
\Big[
\mathrm{Res} \frac{1}{F_a+xy}\frac{dt_1}{t_1}\frac{dt_2}{t_2}dx dy\Big]
\in H^3(Z^{\circ}_a)\, .
$$ 
The counterpart of (iii) is the $A$-hypergeometric system as 
explained in \cite{CKYZ}.
Batyrev \cite{Batyrev} and Stienstra 
\cite{Stienstra} studied the variation of MHS (VMHS) 
on $H^2(\T^2, C_a^{\circ})$ and showed the followings:  
$H^2(\T^2,C^{\circ}_a)\cong \mathbb{C}\omega_0\oplus PH^1(C_a^{\circ})$ 
is isomorphic to a
certain vector space $\mathcal{R}_{F_a}$;
$\omega_0$ and elements obtained by successive applications of the 
Gauss--Manin connection span $H^2(\T^2,C^{\circ}_a)$;
$\omega_0$ satisfies the $A$-hypergeometric system 
considered in \cite{CKYZ}.
For the polyhedron \#1 in Figure \ref{fig:Hirzebruch}, 
Takahashi  \cite{NTakahashi} independently showed that integrals 
\begin{equation}\nonumber
\int_{\Gamma}\omega_0~,
\quad \Gamma \in  H_2 (\T^2, C_a^{\circ}, \mathbb{Z})\, , 
\end{equation}
over  $2$-chains $\Gamma$ whose boundaries  
lie in 
$C^{\circ}_a$
satisfy the same differential equation.  
For the open threefold $Z^{\circ}_a$,
there is a result by Hosono \cite{Hosono} that integrals 
\begin{equation}\nonumber
\int_{\gamma}\omega_a~,\quad
\gamma \in H_3(Z^{\circ}_a, \mathbb{Z})\, ,
\end{equation}
satisfy exactly the same $A$-hypergeometric system. 
It has been known that $H^3(Z_a^{\circ}) \cong H^2(\T^2,C_a^{\circ})$.
Gross \cite[\S 4]{Gross} described the isomorphism and mentioned that
the integration of  $\omega_a$ over a $3$-cycle reduces to
that  of  $\omega_0$ over a $2$-chain under the isomorphism.
In this paper, we shall study the (V)MHS of $H^3(Z^{\circ}_a)$ and show that
it has the same description as $H^2(\T^2,C_a^{\circ})\cong \mathcal{R}_{F_a}$ 
and that $\omega_a$ plays the same role as $\omega_0$. 
This is one of the main results of this paper (cf. Theorem \ref{thm:relationship}). 

\begin{remark}
In \cite{CKYZ}, Chiang et al. considered the
 ``$1$-form''
$
\mathrm{Res}_{F_a=0} (\log F_a)\omega_0
$  on $C^{\circ}_a$ and argued
that its period integrals satisfy the $A$-hypergeometric system.
The result by  Batyrev, Stienstra, and Takahashi implies that
$\omega_0\in H^2(\T^2,C^{\circ}_a)$ 
gives a rigorous definition of this ``$1$-form''. 
This point was mentioned in \cite{Gross}.
\end{remark}

\begin{remark}
Calculation of the (V)MHS of $H^3(Z^{\circ}_a)$
in this paper closely follows the result by Batyrev on 
the MHS of affine hypersurfaces in algebraic tori \cite{Batyrev}.
\end{remark}

\subsection{Weight filtration and the Yukawa coupling}\label{sec:intro-yukawa}
At this point, one may ask what is the role of the weight filtration.
Our answer is that it is needed to define an analogue of the Yukawa coupling.
It is the main motivation of the present work.
In general, 
the lowest level subspace of the weight filtration in $H^*(V^{\circ})$
is the image of the cohomology $H^{*}(V)$ 
of a smooth compactification $V$ 
(see, e.g., \cite[Proposition 6.30]{PS}).
In our cases, it turns out that  
the lowest level subspace
$\mathcal{W}_1 H^2(\T^2,C^{\circ}_a)$
(resp. $\mathcal{W}_3H^3(Z^{\circ}_a)$)
of the weight filtration on
$H^2(\T^2,C^{\circ}_a)$ (resp. $H^3(Z^{\circ})$) 
is isomorphic to $H^1(C_a)$ (resp. $H^3(Z_a)$), 
where $C_a$ (resp. $Z_a$) is a smooth compactification
of $C_a^{\circ}$ (resp. $Z_a^{\circ}$).
Thus we can use the intersection product on $H^1(C_a)$ or $H^3(Z_a)$ to define 
an analogue of the Yukawa coupling. 

As a counterpart to (iv), we propose the following definition (Definition \ref{definition:Yukawa2}).
Consider the family of affine curves $\mathcal{Z}\rightarrow \mathbb{L}_{{\rm reg}}(\Delta)$.
Let $T^0\mathbb{L}_{{\rm reg}}(\Delta)$ be the 
subbundle of the holomorphic tangent bundle $T\mathbb{L}_{{\rm reg}}(\Delta)$
spanned by  $\partial_{a_0}$
\footnote{Here $a_0$ is the parameter corresponds to the origin $(0,0) \in \Delta \cap \mathbb{Z}^2$.}. 
Our Yukawa coupling  is a multilinear map from
$T\mathbb{L}_{{\rm reg}}(\Delta) 
\times
T\mathbb{L}_{{\rm reg}}(\Delta) 
\times T^0\mathbb{L}_{{\rm reg}}(\Delta)$ to $\mathcal{O}_{\mathbb{L}_{{\rm reg}}(\Delta)}$.
Take three vector fields $(A_1,A_2,A_3)\in 
T\mathbb{L}_{{\rm reg}}(\Delta) 
\times
T\mathbb{L}_{{\rm reg}}(\Delta) 
\times T^0\mathbb{L}_{{\rm reg}}(\Delta).
$
By the result on VMHS, we see that 
$\nabla_{A_3}\omega_0$
can be regarded as a $(1,0)$-form on $C_a$,  
and that  
although  
$\nabla_{A_1}\nabla_{A_2}\omega_0$ may not be in $\mathcal{W}_1$, 
we can associate a $1$-form $(\nabla_{A_1}\nabla_{A_2}\omega_0)'$
on $C_a$ (Lemma \ref{prop:6-1}).
We define 
$$
\mathrm{Yuk}(A_1,A_2;A_3)=\sqrt{-1}\int_{C_a} 
(\nabla_{A_1}\nabla_{A_2}\omega_0)'
\wedge \nabla_{A_3} \omega_0~.
$$
It is also possible to define the Yukawa coupling using the family of open
threefolds. In fact they are the same
up to multiplication by a nonzero constant.
We also have a similar definition for the quotient family
(\S \ref{sec:QuotientFamily}). 

In addition to the above geometric definition, 
we give an algebraic description of the Yukawa coupling 
via a certain pairing considered by Batyrev \cite{Batyrev} 
(cf. \S \ref{sec:pairing}, \ref{sec:yukawaRF}).  
We also derive the differential equations for them 
(Proposition \ref{prop:Yukawa-equation1}, Lemma \ref{prop:diffyukawa1}).
These results enables us to compute 
the Yukawa couplings at least in   
the examples shown in Figure \ref{fig:Hirzebruch}
(cf. Example \ref{example:p2yukawa}, \S 8).
They agree with the known results 
\cite{KlemmZaslow, FJ, ABK, BT, HaghihatKlemmRauch, AlimLaengeMayr}.
We also see that they are mapped to  the local A-model Yukawa couplings
by the mirror maps (cf. Example \ref{P2Ayukawa}, \S 8).

\subsection{Local B-model at higher genera}
If we are to pursue further the analogy between 
mirror symmetry and local mirror symmetry to higher genera,
the first thing 
to do is to formulate 
an analogue of the so-called 
special K\"ahler geometry.
It is a K\"ahler metric on 
the moduli $B$ 
of complex deformations of a Calabi--Yau threefold $X^{\vee}$
whose curvature
satisfies a certain equation 
called the special geometry relation.
In the setting of the local B-model,
we define a Hermitian metric on the rank one subbundle
$T^0\moduli$ of $T\moduli$
and derive an equation
similar to the special geometry relation (Lemma \ref{prop:specialgeometry}).

Next, we consider 
Bershadsky--Cecotti--Ooguri--Vafa's (BCOV's) holomorphic 
anomaly equation \cite{BCOV}.
It is 
a partial differential equation 
for the B-model topological string amplitudes $F_g$'s\footnote{
Its mathematical definition is yet unknown for $g\geq 2$.}
which involves the K\"ahler metric, its K\"ahler potential and the Yukawa coupling.
By making use of the result on the VMHS of $H^2(\T^2,C^{\circ}_a)$
(or $H^3(Z_a^{\circ})$), we propose how to adapt
the holomorphic anomaly equation to the setting of  the local B-model
(eqs. \eqref{eq:HAE1}, \eqref{eq:HAE2}). 
We also explain it from 
Witten's geometric quantization approach \cite{Witten}.

In the examples shown in Figure \ref{fig:Hirzebruch},
we checked that 
the solutions of this holomorphic anomaly equation
with appropriate holomorphic ambiguities
and with the holomorphic limit
give the correct local GW invariants for $g=1,2$
at least for small degrees.

\begin{remark}
It is known
that, in the local setting,
the K\"ahler potential drops out 
from BCOV's holomorphic anomaly equation,   
and consequently the equation is solved 
by a certain Feynman rule with only one type of 
propagators $S^{i,j}$ with two indices
\cite{KlemmZaslow, HosonoHAE, ABK,  HaghihatKlemmRauch, AlimLaengeMayr}.
Moreover, it is also known that essentially only one direction in 
$\mathcal{M}(\Delta)$ corresponding to the moduli of 
the elliptic curve $C_a$ is relevant. 
Our description of BCOV's holomorphic anomaly equation 
is based on these results.
\end{remark}
\subsection{Plan of the paper}
In \S \ref{section:Preliminaries}, we recall the definition
of the mixed Hodge structure.
In \S \ref{section:Polyhedron}, we 
define
the vector space $\mathcal{R}_{F_a}$  following Batyrev \cite{Batyrev}
and recall Gel'fand--Kapranov--Zelevinsky's
$A$-hypergeometric system \cite{GKZ1,GKZ2}.
In \S \ref{section:relative},  we give descriptions due to 
Batyrev \cite{Batyrev} and Stienstra \cite{Stienstra} of 
(V)MHS on
the relative cohomology $H^n(\T^n,V_a^{\circ})$
of the pair of the $n$-dimensional algebraic torus $\T^n$
and an affine hypersurface $V_a^{\circ}$
in terms of $\mathcal{R}_{F_a}$ (Theorem \ref{theorem:MHS}). 
In \S \ref{section:threefold}, 
we state the result on the MHS on
the cohomology $H^3(Z_a^{\circ})$ of the open threefold $Z_a^{\circ}$
and its relationship to that on $H^2(\T^2,C_a^{\circ})$ 
(Theorem \ref{thm:relationship}).
The details for the calculation of $H^3(Z_a^{\circ})$ 
are given in \S \ref{section:Threefold}.
In \S \ref{section:Yukawa}, we define the Yukawa coupling 
and study its properties. 
In \S \ref{section:HAE}
we propose a holomorphic anomaly equation for the local B-model.

We note that polyhedra dealt with in
\S \ref{section:Polyhedron} and \S \ref{section:relative}
are convex integral polyhedra,
while in \S \ref{section:threefold}, \S \ref{section:Yukawa} and \S \ref{section:HAE},
only $2$-dimensional reflexive polyhedra are considered.

The examples treated in this article 
are listed in Figure \ref{fig:Hirzebruch}. 
These will be sometimes called the cases
of $\mathbb{P}^2$,
$\mathbb{F}_0$, $\mathbb{F}_1$, $\mathbb{F}_2$
according to their local A-model toric surfaces.
The $\mathbb{P}^2$ case appears in the course of the paper.
The others are summarized in \S \ref{section:Examples}.
\subsection{Notations}
Throughout the paper, we use the following notations. 
$\T^n$ denotes the $n$-dimensional algebraic torus
$(\mathbb{C}^*)^n$. 
For $m=(m_1,\ldots,m_n)\in \mathbb{Z}^n$,
$t^m$ stands for the Laurent monomial
$t_1^{m_1}\cdots t_n^{m_n}\in \mathbb{C}[t_1^{\pm 1},\ldots,t_n^{\pm 1}]$.
For a variable $x$, $\theta_x$ is the 
logarithmic derivative $x{\partial_x}$. 
\subsection{Acknowledgments}
This work was started during Y.K.'s visit at 
Institut des Hautes \'Etudes Scientifiques. She thanks the institute 
for hospitality and support.
She also thanks Maxim Kontsevich, Kyoji Saito, Akihiro Tsuchiya
for valuable comments.
The authors thank Andrea Brini for kindly letting them know
some references on the Yukawa coupling.
Research of Y.K. is partly supported by
Grant-in-Aid for Scientific Research (Start-up 20840024).
S.M. is supported by JSPS/EPDI/IH\'ES
fellowship (2007/2009) and Max-Planck-Instutut f\"ur Mathematik.
\section{Preliminaries on the mixed Hodge structures}\label{section:Preliminaries}
In this section we recall the mixed Hodge structure 
of Deligne \cite{Deligne2, Deligne3}. 
See also \cite{Voisin, Voisin2} \cite{PS}.

A mixed Hodge structure (MHS) $H$ is the triple  
$H=(H_{\mathbb{Z}}, \mathcal{W}_{\bullet}, \mathcal{F}^{\bullet})$, 
where $H_{\mathbb{Z}}$ is a finitely generated abelian group, 
$\mathcal{W}_{\bullet}$ is an increasing filtration 
(called {\it the weight filtration})
on $H_{\mathbb{Q}}=H_{\mathbb{Z}}\otimes \mathbb{Q}$,  
and $\mathcal{F}^{\bullet}$ is a decreasing filtration 
(called {\it the Hodge filtration})
on $H_{\mathbb{C}}=H_{\mathbb{Z}}\otimes \mathbb{C}$ 
such that for each graded quotient 
$$
\text{Gr}_n^{\mathcal{W}}(H_{\mathbb{Q}})
:=\mathcal{W}_n / \mathcal{W}_{n-1} \, ,
$$  
with respect to 
the weight filtration $\mathcal{W}_{\bullet}$, 
the Hodge filtration $\mathcal{F}^{\bullet}$ induces a decomposition 
$$\text{Gr}_n^{\mathcal{W}}(H_{\mathbb{C}})=\bigoplus_{p+q=n}H^{p, q}\, ,
$$ 
with ${\overline{H^{p, q}}}=H^{q, p}$, where 
$H^{p,q}:=\text{Gr}^p_{\mathcal{F}}\text{Gr}_{p+q}^{\mathcal{W}}(H_{\mathbb{C}})$ 
and the bar denotes the complex conjugation. 
We say that the weight $m \in \mathbb{Z}$ 
occurs in $H$
if $\text{Gr}_m^{\mathcal{W}}\neq 0$, and that 
$H$ is a pure Hodge structure of weight $m$ 
if $m$ is the only weight which occurs in $H$.
By the classical Hodge theory, if $X$ is a compact K\"ahler manifold, 
then $H^k(X)$ carries a pure Hodge structure of weight $k$.

Let $V^{\circ}$ be a smooth open algebraic variety of dimension $n$. 
By Deligne \cite{Deligne2, Deligne3}, there is a canonical MHS on $H^k(V^{\circ})$.  
The weights of $H^k(V^{\circ})$ may occur 
in the range $[k, 2k]$ if $k\leq n$ 
and in $[k, 2n]$ if $k\geq n$.  
The construction goes as follows. 
Take a smooth compactification $V$ of $V^{\circ}$ such that 
the divisor $D=V\setminus V^{\circ}$ is simple normal crossing, 
and consider  
meromorphic
differential forms on $V$ 
which may have logarithmic 
poles along the divisor $D$.  
Then the Hodge filtration is given by the degree of logarithmic forms while 
the weight filtration is given by the pole order. 
The constructed MHS does not depend on the chosen compactification $V$ 
and is functorial, i.e., any morphism $f:V^{\circ} \to U^{\circ}$ 
of varieties induces a morphism 
$f^*: H^*(U^{\circ})\to H^*(V^{\circ})$ of MHS's. 

Let $\iota: V^{\circ} \hookrightarrow U^{\circ}$ be an immersion between 
two smooth open algebraic varieties. 
By \cite{Deligne3}, 
there exists a canonical MHS on the relative cohomology 
$H^k(U^{\circ}, V^{\circ})$. 
The construction is similar to the one above 
(cf.  \cite[Ch. 5]{PS}, \cite[Ch. 8]{Voisin2}). 
The long exact sequence 
\begin{equation}\label{eq:pair}
\cdots \overset{\iota^*}{\longrightarrow} H^{k-1}(V^{\circ}) \longrightarrow H^{k}(U^{\circ},V^{\circ})
\longrightarrow H^{k}(U^{\circ}) 
\overset{\iota^*}{\longrightarrow} H^{k}(V^{\circ}) \longrightarrow \cdots\, , 
\end{equation}
is an exact sequence of MHS's. 
The weights of $H^k(U^{\circ}, V^{\circ})$ may occur 
in $[k-1, 2k]$.

The $m$-th Tate structure $T(m)$ is 
the pure Hodge structure of weight $-2m$ 
on the lattice $(2\pi\sqrt{-1})^{m}\mathbb{Z}\subset \mathbb{C}$ 
which is of type $(-m,-m)$, i.e.,  
$T(m)_{\mathbb{C}}=T(m)^{-m, -m}$.

\begin{example}\label{ex:tate}
The MHS on $H^m(\T^n)$ is $T(-m)^{\oplus \binom nm}$ for $0\leq m\leq n$. 
See \cite[Example 3.9]{Batyrev}.
\end{example} 

\begin{example}\label{ex:curve}
Let $C$ be a smooth projective curve and 
$C^{\circ}=C\setminus D$ be an affine curve, 
where $D=\{p_1, \cdots, p_m\}$ is a set of distinct $m$-points on $C$. 
We describe the MHS on $H^1(C^{\circ})$.  
The weight filtration $\mathcal{W}_{\bullet}$ 
and the Hodge filtration $\mathcal{F}^{\bullet}$ 
are: 
$$
0\subset \mathcal{W}_1= H^1(C)\subset \mathcal{W}_2=H^1(C^{\circ})\, ,
\quad
0 \subset \mathcal{F}^1= H^0(\Omega^1_C(\log D)) \subset \mathcal{F}^0=H^1(C^{\circ})\, ,
$$
where $\Omega^1_C(\log D)$ is the sheaf of logarithmic $1$-forms on $(C, D)$.
The ``Hodge decomposition'' of 
${\rm Gr}_1^{\mathcal{W}}\left(H^1(C^{\circ})\right)
=H^{1,0}\oplus H^{0,1}
$ is the same as that of $H^1(C)$ :
$$
H^{1,0}=H^{1,0}(C)\, ,
\quad
H^{0,1}=H^{0,1}(C)\, .
$$ 
That of ${\rm Gr}_2^{\mathcal{W}}\left(H^1(C^{\circ})\right)
=H^{2,0}\oplus H^{1,1}\oplus H^{0,2}$ is,
$$
H^{1,1}={H^0(\Omega^1_C(\log D))} \big/ {H^0(\Omega_C^1)}\, ,
\quad
H^{2.0}=H^{0,2}=0\, .$$
Hence the ``Hodge numbers'' of $H^1(C^{\circ})$ are:
\begin{equation}\nonumber
\begin{array}{r|rccc}
       &p=0&1 &2\\\hline
q=0&     &g &0\\
1     & g&m-1&\\
2     & 0&       &\\ 
\end{array}~~~~,
\end{equation}
where $g$ is the genus of $C$.
The MHS on $H^1(C^{\circ})$ is an  
extension of $T(-1)^{\oplus(m-1)}$ by 
$H^1(C)$ in the sense of \cite{Carlson}:
$$
0\rightarrow H^1(C) \rightarrow H^1(C^{\circ})\rightarrow
T(-1)^{\oplus (m-1)}
\rightarrow 0\, .
$$
\end{example} 
\section{Polyhedron, Jacobian ring, $\mathcal{R}_F$ and $A$-hypergeometric system}
\label{section:Polyhedron}
A convex integral polyhedron $\Delta\subset  \mathbb{R}^n$ 
is the convex hull of some finite set in $\mathbb{Z}^n$.
The set of integral points in $\Delta$ is denoted by
$A(\Delta)$ and 
its cardinality by $l(\Delta):=\#A(\Delta)$.
\subsection{$\Delta$-regularity}
Let $\Delta$ be an $n$-dimensional integral convex polyhedron.
Equip the ring $\mathbb{C}[t_0,t_1^{\pm 1},\ldots, t_n^{\pm 1}]$ 
with the grading given by $\mathrm{det}\,t_0^k t^m =k$.
Define $\ring_{\Delta}$ to be its subring:
$$
\ring_{\Delta} = \bigoplus_{k\geq 0} \ring_{\Delta}^k~,
\qquad
\ring_{\Delta}^k=\bigoplus_{m\in \Delta(k)}\mathbb{C}\, t_0^k t^m~,
$$
where
\begin{equation}\label{def:ktriangle}
\Delta(k):=\Big\{m\in\mathbb{R}^n \mid \frac{m}{k}\in \Delta\Big\}~(k\geq 1)~,\quad
\Delta(0):=\{0\}\subset \mathbb{R}^n~.
\end{equation}

Recall that the Newton polyhedron of a Laurent polynomial 
$$
F=\sum_{m}a_m t^m\in \mathbb{C}[t_1^{\pm 1},\ldots,t_n^{\pm 1}]
$$
is the convex hull 
of $\{ m\in \mathbb{Z}^n \mid a_m\neq 0\}$
in $\mathbb{R}^n$. Denote by $\mathbb{L}(\Delta)$  
the space of Laurent polynomials whose Newton polyhedra are $\Delta$.

\begin{definition}\label{def:regular}
A Laurent polynomial $F$
is said to be $\Delta$-regular if $F\in \mathbb{L}(\Delta)$ and,  
for every $l$-dimensional 
face $\Delta'\subset \Delta$ ($0<l\leq n$), the equations
$$
F^{\Delta'}:=\sum_{m\in\Delta'\cap \mathbb{Z}^n}a_m t^m~=0~,
\quad
\frac{\partial F^{\Delta'}}{\partial t_1}=0~,\ldots, 
\frac{\partial F^{\Delta'}}{\partial t_n}=0~, 
$$
have no common solutions in $\mathbb{T}^n$.
Denote by $\mathbb{L}_{{\rm reg}}(\Delta)$
the space of $\Delta$-regular Laurent polynomials. 
\end{definition}

\begin{example}
Let $\Delta\subset\mathbb{R}^2$ be the 
polyhedron \#1 in Figure \ref{fig:Hirzebruch}, 
which is the convex hull of $\{ (1,0), (0,1), (-1,-1)\}$.  
Let $F \in \mathbb{L}(\Delta)$ which is of the form:
\begin{equation}\label{def:Fex}
F
=a_1t_1+a_2t_2+\frac{a_3}{t_1t_2}+a_0~,
\quad 
(a_0,a_1,a_2,a_3\in \mathbb{C})~.
\end{equation} 
We wrote $a_{(1,0)}=a_1$, $a_{(0,1)}=a_2$, $a_{(-1,-1)}=a_3$
for simplicity.
Then we have 
\begin{equation}\nonumber 
F\in \mathbb{L}_{{\rm reg}}(\Delta)
\, \Longleftrightarrow \, 
\frac{a_0^3}{a_1a_2a_3}+27\neq 0~,\,
a_1a_2 a_3\neq 0~.
\end{equation}
\end{example}
\subsection{Jacobian ring, $\mathcal{R}_F$  and filtrations}
For $F\in \mathbb{L}(\Delta)$,
let $\mathcal{D}_i$ ($i=0,\ldots,n$) be the following
differential operators
acting on $\ring_{\Delta}$:
\begin{equation}\label{def:mathcalD}
\mathcal{D}_{0}:=\theta_{t_0}+t_0 F~,\quad
\mathcal{D}_{i}:=\theta_{t_i}+t_0 \theta_{t_i}F~(i=1,\ldots,n)~.
\end{equation}
\begin{definition}
Define $\mathbb{C}$-vector spaces 
$\mathcal{R}_F$ and $\mathcal{R}^+_{F}$ by 
\begin{equation}
\mathcal{R}_F:=\ring_{\Delta}\Big/
\sum_{i=0}^n \mathcal{D}_i \ring_{\Delta}~,\quad 
\mathcal{R}^+_{F}
:=
\mathbf{S}_{\Delta}^+ \left / 
\sum_{i=0}^n \mathcal{D}_i \mathbf{S}_{\Delta} \right. \, , 
\end{equation}
where $\ring_{\Delta}^+=\sum_{k\geq 1}\ring_{\Delta}^k$.
\end{definition}
Obviously, $\mathcal{R}_F=\mathcal{R}_F^+\oplus \ring_{\Delta}^0$.

We consider two filtrations on the vector spaces $\mathcal{R}_F$. 
The $\mathcal{E}$-filtration on $\ring_{\Delta}$ is a decreasing filtration
\begin{equation}\nonumber 
\mathcal{E} :\cdots
\supset \cdots\supset \mathcal{E}^{-k}\supset \cdots 
\mathcal{E}^{-1}\supset \mathcal{E}^{0}\supset\cdots
\end{equation}
where $\mathcal{E}^{-k}$ is the subspace spanned by all monomials
of the degree $\leq k$.
This induces filtrations on $\mathcal{R}_F$ and $\mathcal{R}^+_{F}$ which are denoted by
$\mathcal{E}$ and $\mathcal{E}_+$ respectively.
It holds that $\mathcal{E}^{-n}=\mathcal{R}_{F}$.

\begin{definition}
Let $J_F$ be the ideal in $\mathbf{S}(\Delta)$ generated by
$t_0F, t_0 \theta_{t_1}F,\ldots, t_0\theta_{t_n}F$. 
The Jacobian ring $R_F$ is defined as $\ring_{\Delta}/J_F$.
Denote by $R_F^i$ the $i$-th homogeneous piece of 
$R_F$.
\end{definition}
The  graded quotient  of $\mathcal{R}_{F}$ 
with respect to the $\mathcal{E}$-filtration is given by
the Jacobian ring: 
$$
\mathrm{Gr}^{-i}_{\mathcal{E}}\mathcal{R}_F=R_F^i\, .
$$

Denote by $
I_{\Delta}^{(j)} (0\leq j\leq n+1)$ 
the homogeneous
ideals in $\ring_{\Delta}$ generated as  $\mathbb{C}$-subspaces
by all monomials $t_0^kt^m$ where  
$k\geq 1$ and $m\in \Delta(k)$
which does not belong to any face of codimension $j$.
We set $I_{\Delta}^{(n+2)}=\ring_{\Delta}$.
These form an increasing chain of ideals in $\ring_{\Delta}$:
\begin{equation}\label{eq:idealI}
0=I_{\Delta}^{(0)}\subset I_{\Delta}^{(1)}\subset I_{\Delta}^{(2)}
\subset \cdots \subset I_{\Delta}^{(n+1)}=\ring_{\Delta}^+
\subset I_{\Delta}^{(n+2)}=\ring_{\Delta}\, .
\end{equation}
Let $\mathcal{I}_j\subset \mathcal{R}_F$ be the image of $I_{\Delta}^{(j)}$.
These subspaces define
an increasing filtration $\mathcal{I}$ on
$\mathcal{R}_F$: 
\begin{equation} \nonumber 
0=\mathcal{I}_0\subset \mathcal{I}_1\subset \mathcal{I}_2\subset\cdots
\subset \mathcal{I}_{n+1}=\mathcal{R}_F^+
\subset 
\mathcal{I}_{n+2}=\mathcal{R}_F~.
\end{equation}
Later we will see that $\mathcal{R}_F$ is isomorphic to 
the cohomology mentioned in \S \ref{sec:intro2}.  
The $\mathcal{I}$-(resp. $\mathcal{E}$-)filtration 
describes the weight (resp. Hodge) filtration 
of the MHS on it. 

\begin{example}
Let $\Delta$  be the 
polyhedron $\#$1 in Figure \ref{fig:Hirzebruch}. 
Assume that $F \in \mathbb{L}_{\rm reg}(\Delta)$. 
Then we have
\begin{equation}\nonumber
\mathcal{R}_F\cong \mathbb{C}\,1\oplus\mathbb{C}\,t_0\oplus
\mathbb{C}\,t_0^2~.
\end{equation}
The $\mathcal{I}$- and the $\mathcal{E}$-filtrations are
\begin{equation}\nonumber
\mathcal{I}_3=\mathcal{I}_2=\mathcal{I}_1\cong \mathbb{C}\,t_0\oplus\mathbb{C}\,t_0^2~,\quad
\mathcal{I}_4=\mathcal{R}_F~.
\end{equation}
\begin{equation}\nonumber
\mathcal{E}^0=\mathbb{C}\,1~,\quad
\mathcal{E}^{-1}=\mathbb{C}\,1\oplus\mathbb{C}\,t_0~,\quad
\mathcal{E}^{-2}=\mathcal{R}_F~.
\end{equation}
\end{example}
\subsection{Derivations with respect to parameters}
Let ${a}=(a_m)_{m\in A(\Delta)}$ be algebraically independent coefficients.
Consider the $\mathbb{C}[a]$-module
 $\ring_{\Delta}[a]:=\ring_{\Delta}\otimes 
\mathbb{C}[a]$. 
Let
\begin{equation}
\mathcal{R}_F[a]:=\ring_{\Delta}[a]\Big/
\Big(\sum_{i=0}^n \mathcal{D}_i \ring_{\Delta}[a]\Big)~.
\end{equation}
Define the action of differential operators
$\mathcal{D}_{a_m}$ ($m\in A(\Delta)$) on $\ring_{\Delta}[a]$
by
\begin{equation}\nonumber 
\mathcal{D}_{a_m}=\frac{\partial}{\partial a_m}
+t_0 t^m~.
\end{equation}
Since this action commutes with that of 
$\mathcal{D}_{i}$ ($i=0,1,\ldots,n$),
it induces an action of $\mathcal{D}_{a_m}$ on $\mathcal{R}_F[a]$.

We shall see that the operator $\mathcal{D}_{a_m}$ corresponds to 
the Gauss--Manin connection
$\nabla_{a_m}$ on the cohomology of our interest 
(cf. \S \ref{sec:GM}, \S \ref{section:threefold}). 
Note that
$\mathcal{D}_{a_m}$ preserves 
the $\mathcal{I}$-filtration: 
$\mathcal{D}_{a_m}\mathcal{I}^{j}\subset \mathcal{I}^{j}$.
This  corresponds to the fact that the weight filtration
is preserved by the variation of MHS's.
Note also that
$\mathcal{D}_{a_m}$ decreases  $\mathcal{E}$-filtration by one:
$\mathcal{D}_{a_m} \mathcal{E}^{-k}\subset \mathcal{E}^{-k-1}$.
This corresponds 
to the Griffiths transversality \cite{Usui}.

\subsection{$A$-hypergeometric system }\label{sec:Ahyper}%
We briefly recall
the $A$-hypergeometric system of Gel'fand--Kapranov--Zelevinsky
\cite{GKZ1} \cite{GKZ2}
in a form suitable to our situation.
Let $\Delta$ be an $n$-dimensional integral convex polyhedron.
For $\Delta$, the lattice of relations is defined by
\begin{equation}\label{eq:lattice-relations}
L(\Delta):=\Big\{
l=(l_m)_{m\in A(\Delta)}\in \mathbb{Z}^{l(\Delta)}\mid
\sum_{m\in A(\Delta)} l_m m=0 ~,~
\sum_{m\in A(\Delta)}l_m=0~
\Big\}~.
\end{equation}
The $A$-hypergeometric system associated to $\Delta$ (with parameters 
$(0, \ldots , 0) \in \mathbb{C}^{n+1}$) is the following system of 
linear differential equations for $\Phi(a)$: 
\begin{equation}\label{Ahyper}
\hyper_i\Phi(a)=0~ (i=0,1,\ldots,n)~,\quad
\square_{\,l}\Phi(a)=0~(l\in L(\Delta))~,
\end{equation}
where
\begin{equation}\nonumber
\begin{split}
&\hyper_0=\sum_{m\in A(\Delta)}\theta_{a_m}~,
\quad
\hyper_i=\sum_{m\in A(\Delta)}m_i\theta_{a_m}
\quad(1\leq i\leq n)~,
\\
&\square_{\,l}=\prod_{\begin{subarray}{c}m\in A(\Delta);\\
 l_m>0\end{subarray}}
\partial_{a_m}^{\,l_m}
 -\prod_{\begin{subarray}{c}m\in A(\Delta);\\l_m<0\end{subarray}}
\partial_{a_m}^{-l_m}~.
\end{split}
\end{equation}
The number of independent solutions
is equal to the volume\footnote{
Here the volume is normalized so that the fundamental simplex in $\mathbb{R}^n$ 
has volume one.}
of the polyhedron $\Delta$
\cite{GKZ1}. 

\begin{example}\label{example:P2Ahyper}
In the case when $\Delta$ is the 
polyhedron $\#$1 in Figure \ref{fig:Hirzebruch},
the lattice of relations $L(\Delta)$ has rank one
and generated by $(-3,1,1,1)$.
For simplicity we write $a_{(1,0)}=a_1$, $a_{(0,1)}=a_2$, $a_{(-1,-1)}=a_3$.
The $A$-hypergeometric system is 
\begin{equation} \nonumber
\begin{split}
&(\theta_{a_1}-\theta_{a_3})\Phi(a)=0~,\quad
(\theta_{a_2}-\theta_{a_3})\Phi( a)=0~,\\
&(\theta_{a_0}+\theta_{a_1}+\theta_{a_2}+\theta_{a_3})\Phi( a)=0~,\\
&(\partial_{a_1}\partial_{a_2}\partial_{a_3}-\partial_{a_0}^3)\Phi(a)=0~.
\end{split}
\end{equation}
It is equivalent to $\Phi(a)=f(z)$, $z=\frac{a_1a_2a_3}{a_0^3}$ and
\begin{equation}\nonumber 
\big[\theta_{z}^3+3{z}\,\theta_{z}(3\theta_{z}+1)(3\theta_{z}+2)\big]f(z)=0~.
\end{equation}
Solutions about $z=0$ are obtained in \cite[eq.(6.22)]{CKYZ}:
\begin{equation}\label{eq:PFsolution}
\begin{split}
&\varpi(z;0)=1~,\quad
t:=\partial_{\rho}\varpi(z;\rho)|_{\rho=0}=\log z+3H(z)~,
\\
&\partial_S F=\partial_{\rho}^2\varpi(z;\rho)|_{\rho=0}=(\log z)^2+\cdots
\end{split}
\end{equation}
where
$$
\varpi(z;\rho)=\sum_{n\geq 0}\frac{(3\rho)_{3n}}
{(1+\rho)_n^3}(-1)^{n}z^{n+\rho}~,\quad
H(z)=\sum_{n\geq 1} \frac{(3n-1)!}{n!^3}(-z)^n~.
$$
Here  $(\alpha)_n$ denotes the Pochhammer symbol:
$(\alpha)_n
=(\alpha)(\alpha+1)\cdots(\alpha+n-1)$ for $n\geq 1$,
$(\alpha)_n=1$ for $n=0$. 
\end{example}

\begin{proposition}\label{prop:Ahyper}
1. For each $F\in \mathbb{L}_{{\rm reg}}(\Delta)$, $\mathcal{R}_F$ is spanned by
$\mathcal{D}_{a_{m_1}}\cdots \mathcal{D}_{a_{m_k}}\,1$
$(0\leq k\leq n$, $m_1,\ldots, m_k\in A(\Delta))$.
\\ 
2. 
In
$\mathcal{R}_F[a]$, the element $1$ satisfies the
$A$-hypergeometric system \eqref{Ahyper} 
with $\partial_{a_m}$ $(m\in A(\Delta))$ replaced by
$\mathcal{D}_{a_m}$.
\end{proposition}
\begin{proof}
1. This follows because
$\ring_{\Delta}$ is spanned by monomials obtained by
successive applications of $\mathcal{D}_{a_m}$ to $1$
and because $\mathcal{R}_{F}=\mathcal{E}^{-n}$.
\\
2. 
In the ring $\ring_{\Delta}[a]$, it holds that
\begin{equation}
\begin{split}
&
\Big(\hyper_i|_{\theta_{a_m}\to a_m\mathcal{D}_{a_m}} \Big)\,1
=\mathcal{D}_i\,1 \quad (0\leq i\leq n)~,
\\&
\Big(\square_{\,l}|_{\partial_{a_m}\to \mathcal{D}_{a_m}}\Big)\,1=
\prod_{m: l_m>0}(t^m)^{l_m}-
\prod_{m:l_m<0}(t^m)^{-l_m}=0~.
\end{split}
\end{equation}
\end{proof}
\subsection{Reflexive polyhedra}\label{sec:reflexive}
Recall from \cite[\S 4]{Batyrev-mirror} the following
\begin{definition} 
An $n$-dimensional convex integral 
polyhedron $\Delta \subset \mathbb{R}^n$ is reflexive  
if $0\in \Delta$ and the distance between $0$ and the hypersurface generated by 
each codimension-one face $\Delta'$ is equal to $1$, i.e., for each codimension-one
face $\Delta'$ of $\Delta$, there exists an integral primitive vector $v_{\Delta'}\in \mathbb{Z}^n$ 
such that $\Delta'=\{m\in \Delta\mid \langle v_{\Delta'}, m\rangle=1\}$.
\end{definition}

In the case when $n=2$,  
it is known that there are exactly sixteen $2$-dimensional 
reflexive polyhedra (see \cite[Fig.1]{CKYZ}).
Let $\Delta$ be a $2$-dimensional reflexive polyhedron 
and $F \in \mathbb{L}_{\rm reg}(\Delta)$. 
Then it is known that
\begin{equation}\nonumber 
\dim \mathrm{Gr}_{\mathcal{E}}^{-k}\mathcal{R}_F=\dim R_F^k=\begin{cases} 1 &(k=0,2)\, ,\\
l(\Delta)-3 &(k=1)\, ,\\
0&(k\geq 3).\end{cases}
\end{equation}
See  Theorem 4.8 in \cite{Batyrev}. 
The $\mathcal{I}$- and the $\mathcal{E}$-filtrations 
on the vector space $\mathcal{R}_{F}$ are described as follows. 
Let $A'(\Delta)=A(\Delta)\setminus \{0, m^{(1)}, m^{(2)},m^{(3)} \}$
where $m^{(1)},m^{(2)},m^{(3)}$  are
any three vertices of $\Delta$.
Then we have
\begin{equation}\label{reflexiveRF}
\begin{split}
\mathcal{R}_{F}=\mathcal{I}_4&=\mathcal{E}^{-2}
\cong \mathbb{C}1\oplus
\bigoplus_{m\in A'(\Delta)}\mathbb{C}t_0t^m \oplus 
\mathbb{C}t_0\oplus \mathbb{C}t_0^2~,
\\
\mathcal{I}_1&\cong \mathbb{C}t_0\oplus \mathbb{C}t_0^2~,
\quad
\mathcal{I}_3\cong\mathcal{I}_1\oplus \bigoplus_{m\in A'(\Delta)}\mathbb{C}t_0t^m~,
\\
\mathcal{E}^{0}&=\mathbb{C}1~,\quad
\mathcal{E}^{-1}\cong\mathcal{E}^0\oplus 
\bigoplus_{m\in A'(\Delta)}\mathbb{C}t_0t^m\oplus \mathbb{C}t_0~.
\end{split}
\end{equation}
As to $\mathcal{I}_2$, it depends on the polyhedron $\Delta$.
For example, $\mathcal{I}_2=\mathcal{I}_1$ for the
polyhedra $\#2,\#3$ in Figure \ref{fig:Hirzebruch}
while $\mathcal{I}_2=\mathcal{I}_3$ for the polyhedra $\# 4$.
See \S \ref{section:Examples}.

For a $2$-dimensional reflexive polyhedron $\Delta$,
there are $l(\Delta)-1$ independent solutions to the $A$-hypergeometric system 
associated to $\Delta$. Explicit expressions for them can be found 
in \cite[eq.(6.22)]{CKYZ}.

\section{Mixed Hodge structures on $H^n(\mathbb{T}^n, V_a^{\circ})$}
\label{section:relative}
Let $\Delta \subset \mathbb{R}^n$ be an 
$n$-dimensional convex integral polyhedron 
and $F_a 
=\sum a_m t^m
\in \mathbb{L}_{\mathrm{reg}}(\Delta)$. 
We denote by $V_a^{\circ}$ the smooth affine hypersurface in $\mathbb{T}^n$ 
defined by $F_a$. 
We state the result on 
the (V)MHS on $H^n(\mathbb{T}^n, V_a^{\circ})$
due to Batyrev \cite{Batyrev} and Stienstra \cite{Stienstra}. 
We remark that $H^k(\mathbb{T}^n, V_a^{\circ})=0$ if $k \neq n$ 
(cf. \cite[Theorem 3.4]{Batyrev}).

The cokernel of the pull-back 
$H^{n-1}(\mathbb{T}^n) \to H^{n-1}(V_a^{\circ})$ 
is called the primitive part of the cohomology of $V_a^{\circ}$ 
and denoted by $PH^{n-1}(V_a^{\circ})$. 
From the long exact sequence (\ref{eq:pair}),  
we obtain the following short exact sequence of MHS's:
\begin{equation}\label{eq:ext}
0 \longrightarrow PH^{n-1}(V_a^{\circ})
\longrightarrow H^n(\mathbb{T}^n, V_a^{\circ})
\longrightarrow H^n(\mathbb{T}^n)
\longrightarrow 0\, .
\end{equation}
Recall that the MHS on $H^n(\mathbb{T}^n)$ is the Tate structure $T(-n)$
(cf. Example \ref{ex:tate}). 
Therefore $H^n(\mathbb{T}^n, V_a^{\circ})$ is an extension of 
$T(-n)$ by $PH^{n-1}(V_a^{\circ})$. 
\subsection{MHS on the primitive part $PH^{n-1}(V_a^{\circ})$}
Let $R^+: S_{\Delta}^+ \to \Gamma\Omega^{n-1}_{V_a^{\circ}}$ be the linear map given by
$$R^+(t_0^k t^m) ={\mathrm{Res}}_{V_a^{\circ}}
\left( \frac{(-1)^{k}(k-1)! \cdot t^m}{F^k} \omega_0\right)\,, 
\quad 
\omega_0:=
\frac{dt_1}{t_1}\wedge \cdots \wedge \frac{dt_n}{t_n}\,.
$$
\begin{theorem}[Batyrev]\label{thm:batyrev}
(i) $R^+$ induces an isomorphism 
\begin{equation} \nonumber 
\rho^+: 
\mathcal{R}^+_{F_a}
\overset{\cong}{\longrightarrow} PH^{n-1}(V_a^{\circ})\,.
\end{equation}
(ii) 
Let $\mathcal{W}_{\bullet}^+$
be the weight filtration on $PH^{n-1}(V_a^{\circ})$. 
Then, for $0< i \leq n-1$, we have 
$$
\rho^+(\mathcal{I}_i)=\mathcal{W}_{n-2+i}^+\, .
$$
(iii) 
Let $\mathcal{F}^{\bullet}_{+}$ be the Hodge filtration on $PH^{n-1}(V_a^{\circ})$.
Then, for $0\leq i \leq n-1$, we have
$$\rho^+(\mathcal{E}^{i-n}_+)=\mathcal{F}^i_+\, .$$
\end{theorem}
\subsection{MHS on the middle relative cohomology 
$H^n(\mathbb{T}^n, V_a^{\circ})$}
Let $R^0 : S_{\Delta}^0 \to \Gamma\Omega_{\mathbb{T}^n}^n$ 
be the linear map given by $R^0(1)=\omega_0$.
Consider the map 
$R:=R^+ \oplus R^0 : 
\mathbf{S}_{\Delta} \to \Gamma\Omega_{\mathbb{T}^n}^n \oplus  \Gamma\Omega^{n-1}_{V_a^{\circ}}$.
Then the following theorem 
follows from Theorem \ref{thm:batyrev} (cf. \cite[Theorem 7]{Stienstra}).

\begin{theorem}[Batyrev, Stienstra]\label{theorem:MHS}
(i) $R$ induces an isomorphism 
\begin{equation}\nonumber 
\rho: 
\mathcal{R}_{F_a}
\overset{\cong}{\longrightarrow} H^{n}(\mathbb{T}^n, V_a^{\circ})\,.
\end{equation}
(ii) 
Let $\mathcal{W}_{\bullet}$
be the weight filtration on $H^{n}(\mathbb{T}^n, V_a^{\circ})$. 
Then we have 
$$
\rho(\mathcal{I}_i)=\mathcal{W}_{n-2+i}\quad(0< i \leq n-1)~, 
\quad 
\rho(\mathcal{I}_{n+1})= \mathcal{W}_{2n-2}=\mathcal{W}_{2n-1}~,
\quad
H^n(\T^n,V_a^{\circ})=\mathcal{W}_{2n}~.
$$
(iii) 
Let $\mathcal{F}^{\bullet}$ be the Hodge filtration on $H^{n}(\mathbb{T}^n, V_a^{\circ})$.
Then, for $0\leq i \leq n$, we have
$$\rho(\mathcal{E}^{i-n})=\mathcal{F}^i\, .$$
\end{theorem}
\subsection{Gauss--Manin connection on $H^n(\mathbb{T}^n, V_a^{\circ})$}
\label{sec:GM}
Consider the variation of MHS on  $H^n(\mathbb{T}^n, V_a^{\circ})$
over $\mathbb{L}_{\mathrm{reg}}(\Delta)$. 
It was studied by Stienstra \cite[\S 6]{Stienstra}.
\begin{lemma}[Stienstra]\label{prop:GMrelative}
The Gauss--Manin connection $\nabla_{\frac{\partial}{\partial a_i}}$ on 
$H^n(\mathbb{T}^n, V_a^{\circ})$ corresponds to 
the operator $\mathcal{D}_{a_i}$ on $\mathcal{R}_F[\mathbf{a}]$.  
\end{lemma}

Stienstra proved this by considering the de Rham complex 
$(\Omega^{\bullet}(\T^n),d+dF_a\wedge )$
\cite[\S 6]{Stienstra}.
Here we give a different proof.
This is a generalization of Takahashi's argument \cite[Lemma 1.8]{NTakahashi}.
\begin{proof}
Since $F_a$ is $\Delta$-regular, there exists 
a holomorphic $(n-1)$-form $\psi_a$ in an open neighborhood of 
$V_a^{\circ}$ in $\mathbb{T}^n$ such that $\omega_0=dF_a \wedge \psi_a$. 
The restriction of $\psi_a$ to $V_a^{\circ}$ is equal to $\text{Res}_{V_a^{\circ}} \frac{\omega_0}{F_a}$ 
and is denoted by $\frac{\omega_0}{dF_a}$. 
It is called the Gelfand--Leray form of $\omega_0$ (cf. \cite[Ch. 10]{AGV}).  

Let $\Gamma_a \in H_n(\mathbb{T}^n, V_a^{\circ})$.  
Then one can show that  
$$
\frac{\partial}{\partial a_i}
\int_{\Gamma_a} \omega_0 
= - \int_{\partial \Gamma_a} 
\frac{\partial F_a}{\partial a_i} \frac{\omega_0}{dF_a} \, .
$$
Namely, we have
$\nabla_{\frac{\partial}{\partial a_i}} \rho(1)=\rho\left(\mathcal{D}_{a_i}1\right)$. 
Since Batyrev \cite{Batyrev} has shown that the Gauss--Manin connection 
$\nabla_{\frac{\partial}{\partial a_i}}$ on $PH^{n-1}(V_a^{\circ})$ 
corresponds to $\mathcal{D}_{a_i}$ under $\rho^+$,  
the lemma follows.
\end{proof}

Lemma \ref{prop:GMrelative} and Proposition \ref{prop:Ahyper}
imply the following 
\begin{corollary}[Stienstra]
1. $H^n(\mathbb{T}^n, V_a^{\circ})$ is spanned by 
$\nabla_{a_{m_1}}\cdots\nabla_{a_{m_k}}\omega_0$
$(0\leq k\leq n$, $m_1,\ldots,m_k\in A(\Delta))$.
\\
2. 
$\omega_0$ satisfies the $A$-hypergeometric system \eqref{Ahyper}
with $\partial_{a_i}$ replaced by $\GM_{\partial_{a_i}}$.
\\
3. 
The period integrals $\int_{\Gamma_a} \omega_0$ 
of the relative cohomology $H^n(\mathbb{T}^n, V_a^{\circ})$ 
satisfies the $A(\Delta)$-hypergeometric system 
\eqref{Ahyper}.
Conversely, a solution of the $A$-hypergeometric system
\eqref{Ahyper} is a period integral.
\end{corollary}
\section{Mixed Hodge structure on $H^3(Z^{\circ}_a)$}
\label{section:threefold}
Throughout the section,  $\Delta$ is a $2$-dimensional reflexive polyhedron 
and $F_a\in \mathbb{L}_{\rm reg}(\Delta)$ is a $\Delta$-regular Laurent polynomial.
\subsection{MHS on the cohomology of the threefold}
Consider the affine threefold $Z^{\circ}_a$ defined by
\begin{equation}\label{eq:defZ}
Z^{\circ}_a=\{(t,x,y)\in \T^2 \times \mathbb{C}^2\mid
F_a(t)+xy=0\}\, .
\end{equation}
We compute $H^3(Z^{\circ}_a)$ and its (V)MHS
following Batyrev \cite{Batyrev}.
Let us briefly state the results. Details are relegated to 
\S \ref{section:Threefold}.

The Poincar\'e residue map
$\mathrm{Res}: H^4(\T^2\times \mathbb{C}^2\setminus Z^{\circ}_a)
\stackrel{\cong}{\rightarrow} H^3(Z^{\circ}_a)$
is an isomorphism (see eq. \eqref{eq:PoincareResidueMap}). 
By Grothendieck \cite{Grothendieck},
$H^{\bullet}(\T^2\times \mathbb{C}^2\setminus Z_a^{\circ})$ 
is isomorphic to the cohomology of
the global de Rham complex 
$(\Gamma\Omega^{\bullet}_{\T^2\times \mathbb{C}^2}(*Z_a^{\circ}),d)$
of meromorphic differential forms on $\T^2\times\mathbb{C}^2$ with poles of
arbitrary order on $Z^{\circ}_a$.
We can show that the  homomorphism:
$$R': 
\ring_{\Delta}\to \Gamma\Omega^4_{\T^2\times\mathbb{C}^2}(*Z_a^{\circ})~;\quad
t_0^kt^m\mapsto 
\frac{(-1)^kk!\,t^m}{(F_a+xy)^{k+1}}
\frac{dt_1}{t_1}\frac{dt_2}{t_2}dxdy~,
$$
induces an isomorphism
$
\mathcal{R}_{F_a}\stackrel{\cong}{\rightarrow} H^4(\T^2\times \mathbb{C}^2\setminus Z^{\circ}_a).
$
Together with the residue map,
we obtain an isomorphism $\rho':\mathcal{R}_{F_a}\overset{\cong}{\rightarrow} H^3(Z^{\circ}_a)$.
The Gauss--Manin connection $\nabla_{\partial_{a_m}}$ on $H^3(Z^{\circ}_a)$
corresponds to a differentiation by $a_m$ on $\Gamma\Omega^4_{\T^2\times\mathbb{C}^2}(*Z^{\circ}_a)$, which in turn corresponds to
the derivation $\mathcal{D}_{a_m}$ on $\mathcal{R}_{F_a}$ (\S \ref{sec:A6}). 

To compute the weight and the Hodge filtrations,
we compactify $Z_a^{\circ}$ as a smooth hypersurface in a
toric variety (\S \ref{sec:compact}). Then we can work out 
calculation similar to \cite[\S 6,\S 8]{Batyrev}.
(Since our $Z^{\circ}_a$ is a hypersurface in $\T^2\times \mathbb{C}^2$,
not in $\T^4$,
we need some modifications. 
Especially we need Mavlyutov's results
on Hodge numbers of semiample hypersurfaces in a toric varieties \cite{Mavlyutov}.)
It turns out that 
the weight and the Hodge filtrations are given by
the $\mathcal{I}$ and the $\mathcal{E}$-filtrations on $\mathcal{R}_{F_a}$.
The result on MHS of $H^3(Z^{\circ}_a)$ is summarized as follows (Theorems
\ref{prop:Hodge}, \ref{prop:weight2}).
\begin{equation}
\begin{split}
&H^3(Z^{\circ}_a)=\mathcal{W}_6=
\mathcal{F}^0=\mathcal{F}^1\cong \mathcal{R}_{F_a}~,
\\
&\mathcal{W}_3\cong \mathcal{I}_1~,
\quad
\mathcal{W}_4=\mathcal{W}_5\cong \mathcal{I}_3~,
\\
&\mathcal{F}^2\cong \mathcal{E}^{-1}~,\quad
\mathcal{F}^3\cong \mathcal{E}^{0}~.
\end{split}
\end{equation}
\subsection{Relationship to the relative cohomology}
Let $C_a^{\circ}$ be the affine curve in $\T^2$ defined by $F_a$.
Since $\Delta$ is reflexive, it is an affine elliptic curve 
obtained by deleting $l(\Delta)-1$ points from an elliptic curve $C_a$. 
The MHS on the primitive part 
$PH^1(C_a^{\circ})$ is an extension of 
$T(-1)^{\oplus (l(\Delta)-4)}$ by $H^1(C_a)$: 
$$
0\rightarrow H^1(C_a) \rightarrow PH^1(C_a^{\circ})
\rightarrow T(-1)^{\oplus(l(\Delta)-4)}
\rightarrow 0
\, .
$$
This follows from the definition of primitive part and 
the description of $H^1(C_a^{\circ})$ given in Example \ref{ex:curve}.
The MHS on the relative cohomology $H^2(\T^2, C_a^{\circ})$
is an extension \eqref{eq:ext} of $H^2(\T^2)=T(-2)$ by 
$PH^1(C_a)$ which turns out to be the trivial one
(cf. Theorem \ref{theorem:MHS}).

Let 
$\rho: \mathcal{R}_{F_a}\overset{\cong}{\rightarrow} H^2(\T^2, C_a^{\circ})$ 
be the isomorphism in Theorem \ref{theorem:MHS} and 
$\omega_0=\rho(1)\in H^2(\T^2, C_a^{\circ})$.
Let 
$$
\omega_a=\rho' (1)
=\Big[\mathrm{Res}\frac{1}{F_a+xy}
\frac{dt_1}{t_1}\frac{dt_2}{t_2}{dx}{dy}\Big] \in H^3(Z_a^{\circ})\,.
$$

\begin{theorem}\label{thm:relationship} 
The composition of isomorphisms
$H^3(Z^{\circ}_a)
\overset{\rho'^{-1}}{\rightarrow} \mathcal{R}_{F_a}
\overset{\rho}{\rightarrow} H^2(\mathbb{T}^2, C_a^{\circ})$
gives
an isomorphism 
$$ \rho \circ \rho'^{-1} : H^3(Z^{\circ}_a) \overset{\cong}{\rightarrow} H^2(\mathbb{T}^2, C_a^{\circ})$$
of $\mathbb{C}$-vector spaces which sends $\omega_a$ to $\omega_0$.
The filtrations correspond as follows: 
$$
\mathcal{F}^{i+1}H^3(Z^{\circ}_a) \overset{\cong}{\rightarrow}  
\mathcal{E}^{i-2} \overset{\cong}{\rightarrow} 
\mathcal{F}^iH^2(\mathbb{T}^2, C_a^{\circ})
\quad (i=0,1, 2)\, ,
$$
$$
\mathcal{W}_{3}H^3(Z^{\circ}_a) \overset{\cong}{\rightarrow}  \mathcal{I}_{1} 
\overset{\cong}{\rightarrow}\mathcal{W}_{1}H^2(\mathbb{T}^2, C_a^{\circ})\, ,
$$
$$
\mathcal{W}_{4}H^3(Z^{\circ}_a) =\mathcal{W}_{5}H^3(Z^{\circ}_a) 
\overset{\cong}{\rightarrow}  \mathcal{I}_{3} 
\overset{\cong}{\rightarrow}
\mathcal{W}_{2}H^2(\mathbb{T}^2, C_a^{\circ})=
\mathcal{W}_{3}H^2(\mathbb{T}^2, C_a^{\circ})\, .
$$
Moreover, $\rho\circ\rho'^{-1}$ is compatible with the Gauss--Manin
connections.
\end{theorem}

Note that $\mathcal{W}_1H^2(\T^2, C_a^{\circ})
=\mathcal{W}_1PH^1(C_a^{\circ})
\cong H^1(C_a)$. 
Therefore, it inherits a nondegenerate pairing. 
The same is true for $\mathcal{W}_3H^3(Z_a^{\circ})$, since it is isomorphic 
to the cohomology $H^3(Z_a)$ of a certain smooth compactification 
$Z_a$ of $Z_a^{\circ}$ (cf. \S  \ref{sec:compact})%
\footnote{
The divisor 
$Z_a\setminus Z_a^{\circ}$ is not smooth but simple normal crossing.
The pull-back $H^3(Z_a)\to \mathcal{W}_3H^3(Z_a^{\circ})$, 
which is always surjective, turns out to be injective. 
This can be checked by comparing the dimension given in 
Lemma \ref{lem:hodgenumber}
and that in Proposition \ref{prop:weight3}.
}. 
  
\section{Analogue of Yukawa coupling}\label{section:Yukawa}
In this section, $\Delta$ is a $2$-dimensional reflexive polyhedron
unless otherwise specified. 
\subsection{Definition of Yukawa coupling via  
affine curves or threefolds}
Let $\Delta$ be a $2$-dimensional reflexive polyhedron.
Let $T^0\mathbb{L}_{{\rm reg}}(\Delta)$ be
the subbundle of the holomorphic tangent bundle  
$T\mathbb{L}_{{\rm reg}}(\Delta)$
of $\mathbb{L}_{{\rm reg}}(\Delta)$
generated by $\partial_{a_0}$.
Consider the family of affine elliptic curves
$p:\mathcal{Z}\to \mathbb{L}_{{\rm reg}}(\Delta)$:
$$
\mathcal{Z}=\{(a,t)\in \mathbb{L}_{{\rm reg}}(\Delta)\times \mathbb{T}^2\mid
F_a(t)=0\}~.
$$
Let $C_a $ be the smooth compactification
of the affine curve $C^{\circ}_a:=p^{-1}(a)$.
Note that we have
$\mathrm{Gr}^0_{\mathcal{F}}H^2(\T^2,C^{\circ}_a)=
\mathrm{Gr}^0_{\mathcal{F}} \mathcal{W}_1 H^2(\T^2,C_a^{\circ})
$.

\begin{lemma}\label{prop:6-1}
For any $\alpha\in H^2(\T^2, C^{\circ}_a)$,
there exists $\alpha' \in \mathcal{W}_1H^2(\T^2,C^{\circ}_a)(\cong H^1(C_a))$
such that
$
[\alpha]=[\alpha']
$
in $\mathrm{Gr}^0_{\mathcal{F}}H^2(\T^2,C^{\circ}_a)=
\mathrm{Gr}^0_{\mathcal{F}} \mathcal{W}_1 H^2(\T^2,C_a^{\circ})$.
\end{lemma}
\begin{proof}
By \eqref{reflexiveRF},
$\alpha$ is written as 
$$
\alpha=\alpha_{2,0}\rho(t_0^2)+\sum_{m\in A'(\Delta)}\alpha_{1,m} \rho(t_0t^m)+
\alpha_{1,0}\rho(t_0)+\alpha_{0,0}\omega_0~ .
$$
Take $\alpha'=\alpha_{2,0}\rho(t_0^2)+c \rho(t_0)$, where $c \in \mathbb{C}$
is arbitrary.
\end{proof}

The pairing 
$$
H^2(\T^2,C^{\circ}_a)\times \mathcal{F}^1 \mathcal{W}_1 H^2(\T^2,C^{\circ}_a)
\rightarrow \mathbb{C}\, ; \quad
(\alpha,\beta)\mapsto\int_{C_a}\alpha'\wedge \beta
$$ 
is independent of the choice of $\alpha'$. 
Recall that $\nabla_{a_0}\omega_0 \in \mathcal{F}^1 \mathcal{W}_1
H^2(\T^2,C^{\circ}_a)$.

\begin{definition}\label{definition:Yukawa2}
For $k\geq 1$,
we define a map
$$\mathrm{Yuk}^{(k)}:
\underbrace{T\mathbb{L}_{{\rm reg}}(\Delta) 
\times \cdots\times
T\mathbb{L}_{{\rm reg}}(\Delta)}_{(k-1)\text{ times }}
\times T^0\mathbb{L}_{{\rm reg}}(\Delta)
\to \mathcal{O}_{\mathbb{L}_{{\rm reg}}(\Delta)}
$$
by 
$$
\mathrm{Yuk}^{(k)}(A_1,\ldots,A_{k-1};A_k)
=\int_{C_a} (\nabla_{A_1}\cdots \nabla_{A_{k-1}} \omega_0)'\wedge
\nabla_{A_k} \omega_0~.
$$
We call $\mathrm{Yuk}^{(3)}$ the Yukawa coupling and 
denote it by $\mathrm{Yuk}$.
\end{definition}

\begin{remark} 
$
\mathrm{Yuk}^{(1)}=\mathrm{Yuk}^{(2)}=0$
by Griffiths' transversality.
For $k\geq 4$, 
$\mathrm{Yuk}^{(k)}(A_1,\ldots, A_{k-1};A_k)$ is 
$\mathcal{O}_{\mathbb{L}_{{\rm reg}}(\Delta)}$-linear
in $A_1,A_k$ and 
$\mathbb{C}$-linear in $A_2,\ldots, A_{k-1}$. 
For $k=3$,  
$\mathrm{Yuk}^{(3)}$ is $\mathcal{O}_{\mathbb{L}_{{\rm reg}}(\Delta)}$-multilinear.
\end{remark}

\begin{remark}
Instead of the relative cohomology $H^2(\T^2,C^{\circ}_a)$,
we can 
use the cohomology $H^3(Z^{\circ}_a)$  of 
the open threefold $Z^{\circ}_a$ defined in \eqref{eq:defZ}, 
provided that the levels of Hodge and
weight filtrations are shifted
according to Theorem \ref{thm:relationship}
and that the integration on $C_a$ 
is replaced by that on the compact threefold $Z_a$ defined in \S  \ref{sec:compact}.
\end{remark}
\subsection{Batyrev's paring}\label{sec:pairing}
We would like to give an algebraic description 
of the Yukawa coupling in terms of the Jacobian ring ${R}_{F_a}$. 
For that purpose, we recall Batyrev's pairing \cite[\S 9]{Batyrev}.
Let $\Delta$ 
be an integral  convex $n$-dimensional polyhedron
and  
$F_a\in \mathbb{L}_{{\rm reg}}(\Delta)
$ a $\Delta$-regular Laurent polynomial.
Denote by $D_{F_a}$ the quotient 
$$D_{F_a}:=
I_{\Delta}^{(1)}\Big/
 (t_0F_a,t_0\theta_{t_1}F_a,\ldots, t_0 \theta_{t_n}F_a )
   \cdot I_{\Delta}^{(1)}~.
$$
It is a graded $R_{F_a}$-module 
consisting of the homogeneous pieces $D_{F_a}^i$ ($1\leq i\leq n+1$).
We have $D_{F_a}^{n+1}\cong \mathbb{C}~.$
The multiplicative structure of $R_{F_a}$-module
defines a nondegenerate pairing
$$
\langle~~,~~\rangle:
R_{F_a}^i\times D_{F_a}^{n+1-i}\rightarrow D_{F_a}^{n+1}
\cong \mathbb{C}~.
$$
Let $H_{F_a}$ be the image of the homomorphism
$
D_{F_a}\to R_{F_a}$ induced by the inclusion $I_{\Delta}^{(1)}\hookrightarrow \ring_{\Delta}$.
Then the above pairing induces a nondegenerate pairing
\begin{equation}\nonumber
\{~~,~~\}:H_{F_a}^i\times H_{F_a}^{n+1-i}\to D_{F_a}^{n+1} \cong \mathbb{C};~~
\{\alpha,\beta\}:=\langle \alpha,\beta' \rangle\, ,
\end{equation}
where $\beta'\in D_{F_a}^{n+1-i}$ 
is an element such that its image
by the homomorphism $D_{F_a}^{n+1-i} \to H_{F_a}^{n+1-i}$
is $\beta$.
\subsection{Yukawa coupling in terms of Batyrev's pairing}
\label{sec:yukawaRF}
Now we come back  
to the case when $\Delta$ is a $2$-dimensional
reflexive polyhedron. In this case, we have $D_{F_a}\cong t_0 R_{F_a}$.
We
explain that the Yukawa coupling defined in Definition \ref{definition:Yukawa2}
is essentially Batyrev's pairing together with 
a choice (concerning the dependence on the parameter $a$) 
of the isomorphism 
$$
\xi_a: D_{F_a}^3\rightarrow \mathbb{C}~.
$$

First identify $\mathcal{I}_1$
with $H_{F_a}$ so that it is compatible with
the Hodge decomposition $H^1(C_a)=H^{1,0}(C_a)\oplus H^{0,1}(C_a)$
under the isomorphism $\rho:\mathcal{R}_{F_a}\to H^2(\T^2,C_a^{\circ})$.
Then Batyrev's pairing 
\begin{equation}\label{HFpairing}
\{~~,~~\}:~H_{F_a}^2\times H_{F_a}^1 \to D_{F_a}^3\cong \mathbb{C},
\end{equation}
induces an antisymmetric pairing
$\langle~~,~~ \rangle_{\mathcal{I}_1}$
on $\mathcal{I}_1$.
Although we do not have an explicit description of
such decomposition $\mathcal{I}_1= H_{F_a}^1\oplus H_{F_a}^2$,
the fact that $H_{F_a}^1$ and $H_{F_a}^2$ are one-dimensional
makes it possible to find $\langle ~~,~~\rangle_{\mathcal{I}_1}$%
\footnote{
An
isomorphism $\mathcal{I}_1\to 
H_{F_a}^1\oplus H_{F_a}^2$ compatible with the graded quotient
is given by
$\alpha_{1,0} t_0+\alpha_{2,0} t_0^2\mapsto 
 (\alpha_{1,0}-u)t_0\oplus \alpha_{2,0} t_0^2$
with some $u$.
The induced antisymmetric pairing on $\mathcal{I}_1$
turns out to be independent of $u$.
}.
It is given by
$$
\langle \alpha_{1,0} t_0+\alpha_{2,0} t_0^2\,,\,
\beta_{1,0} t_0+\beta_{2,0} t_0^2\rangle_{\mathcal{I}_1}
=(-\alpha_{1,0}\beta_{2,0}+\alpha_{2,0}\beta_{1,0})\, \xi_a(t_0^3)~.
$$
Our choice of
$\xi_a$ is as follows.
\begin{proposition}\label{prop:flat}
There exists a map $\xi_a: D_{F_a}^3\rightarrow \mathbb{C}$ which 
is holomorphic in $a \in \mathbb{L}_{{\rm reg}}(\Delta)$ 
and satisfies the following condition: 
\begin{equation}\label{eq:xi-condition}
\langle\mathcal{D}_{a_m}\alpha,\beta\rangle_{\mathcal{I}_1}+
\langle\alpha,\mathcal{D}_{a_m}\beta\rangle_{\mathcal{I}_1}=
\partial_{a_m}\langle\alpha,\beta \rangle_{\mathcal{I}_1}~.
\end{equation}
\end{proposition}
\begin{proof}
Define 
$\alpha_m,\, \beta_m\in \mathbb{C}(a)$ ($m\in A(\Delta)$) 
and $\gamma,\, \delta\in \mathbb{C}(a)$ by
the following relations in $\mathcal{I}_1$:
$$
t_0^2t^m=\alpha_m t_0+\beta_m t_0^2\,,\quad
t_0^3=\gamma t_0+\delta t_0^2\, .
$$
Then the condition \eqref{eq:xi-condition} 
is equivalent to
\begin{equation} \label{eq:xi-condition2}
\partial_{a_m}\xi_a(t_0^3) =-(2\alpha_m+\delta \beta_m+\partial_{a_0}\beta_m)\xi_a (t_0^3)~.
\end{equation}
The existence of a solution $\xi_a(t_0^3)$ to this equation is ensured by the equation
$$
\partial_{a_n}(2\alpha_m+\delta \beta_m+\partial_{a_0}\beta_m)
=\partial_{a_m}(2\alpha_n+\delta \beta_n+\partial_{a_0}\beta_n),
$$
which follows from
the relations in $\mathcal{I}_1$:
$$
\mathcal{D}_{a_n}t_0^2t^m-\mathcal{D}_{a_m}t_0^2t^n=0\, ,\quad
\mathcal{D}_{a_n}t_0^3t^m-\mathcal{D}_{a_m}t_0^3t^n=0\, .
$$
\end{proof}
\begin{remark}\label{remark:xi-condition}
The condition \eqref{eq:xi-condition} 
is equivalent to the following equation for the 
intersection product on $H^1(C_a)$
under the isomorphism $\rho:\mathcal{R}_{F_a}\to H^2(\T^2,C^{\circ}_a)$:
$$
\int_{C_a} \nabla_{a_m}\alpha\wedge \beta
+\int_{C_a}\alpha \wedge \nabla_{a_m}\beta=
\partial_{a_m}\int_{C_a}\alpha\wedge \beta~,
$$
which is well-known in the context of variations of 
polarized Hodge structures. 
\end{remark}
\begin{example}
For the polyhedron $\#$1 in Figure \ref{fig:Hirzebruch},
solving \eqref{eq:xi-condition2},
we obtain
$$
\xi_a(t_0^3)=\frac{1}{27a_1a_2a_3+a_0^3}\times \text{a nonzero constant}~.
$$
\end{example}

Batyrev's pairing \eqref{HFpairing}
together with 
the quotient map $\mathcal{R}_{F_a}\to H_{F_a}^2=R_{F_a}^2=
\mathcal{E}^{-2}/\mathcal{E}^{-1}$ 
induces a pairing
$$
(~~,~~): \mathcal{R}_{F_a}\times H_{F_a}^1\to D_{F_a}^3
\stackrel{\xi_a}{\cong} \mathbb{C}~.
$$
Then, by Remark  \ref{remark:xi-condition},
we have the equation
\begin{equation}\label{eq:equality}
\mathrm{Yuk}^{(k)}(A_1,\ldots, A_{k-1};A_{k})=
(\mathcal{D}_{A_1}\cdots
\mathcal{D}_{A_{k-1}}1 ,\mathcal{D}_{A_k} 1)
\times \text{ a nonzero constant }.
\end{equation}
Here
$\mathcal{D}_A$ is the shorthand notation for
$$
\mathcal{D}_{A}:=\sum_{m\in A(\Delta)}A_m\mathcal{D}_{a_m}~$$
where
$A=\sum_{m\in A(\Delta)}A_m\partial_{a_m}$ is a vector field
on $\mathbb{L}_{{\rm reg}}(\Delta)$.

\begin{example}\label{p2yukawa1}
Let $\Delta$ be the polyhedron $\#$1 in Figure \ref{fig:Hirzebruch}. 
By \eqref{eq:equality}, 
the Yukawa coupling 
$\mathrm{Yuk}(\partial_{a_0},\partial_{a_0};\partial_{a_0})$
is equal to
$
(\mathcal{D}_{a_0},\mathcal{D}_{a_0}1\,,\,\mathcal{D}_{a_0}1)=\xi_a(t_0^3)
$
up to non-zero multiplicative constant. Compare with Example \ref{example:p2yukawa} below.
\end{example}
\subsection{Yukawa coupling and the $A$-hypergeometric system}
Recall 
the $A$-hypergeometric system introduced in \S \ref{sec:Ahyper}. 
The following proposition enables us to compute the Yukawa coupling
by the $A$-hypergeometric system. (See also Lemma \ref{prop:diffyukawa1} in 
the next subsection.)

\begin{proposition}\label{prop:Yukawa-equation1}
1.  For $k\geq 3$ and $m_1,\ldots,m_{k-1}\in A(\Delta)$,
\begin{equation}\label{eq:yukawa1}
\mathcal{T}_i
\mathrm{Yuk}^{(k)}(\theta_{a_{m_1}},\ldots,\theta_{a_{m_{k-1}}};
\theta_{a_0})=0 \quad (i=0,1,2,3).
\end{equation}
2. For a vector $l=(l_m)_{m\in A(\Delta)}\in L(\Delta)$,
let $k$ be the order of the differential operator $\square_l$.
Let us write $\square_l$ as
$$\square_l=\partial_{a_{m_1}}\cdots\partial_{a_{m_k}}-
\partial_{a_{n_1}}\cdots\partial_{a_{n_k}}~.
$$
Then we have
\begin{equation} \nonumber 
\mathrm{Yuk}^{(k+1)}(\partial_{a_{m_1}},\ldots,\partial_{a_{m_k}};\partial_{a_0})-
\mathrm{Yuk}^{(k+1)}(\partial_{a_{n_1}},\ldots,\partial_{a_{n_k}};\partial_{a_0})=0~.
\end{equation}
Moreover, for $j_1,\ldots,j_h\in A(\Delta)$, we have
$$
\mathrm{Yuk}^{(k+h+1)}(\partial_{a_{j_1}},\ldots,\partial_{a_{j_h}},
\partial_{a_{m_1}},\ldots,\partial_{a_{m_k}};\partial_{a_0})-
\mathrm{Yuk}^{(k+h+1)}(\partial_{a_{j_1}},\ldots,\partial_{a_{j_h}},\partial_{a_{n_1}},\ldots,\partial_{a_{n_k}};\partial_{a_0})=0~.
$$ 
3. For $m,n\in A(\Delta)$,
\begin{equation}\nonumber 
\partial_{a_m}\mathrm{Yuk}^{(3)}(\partial_{a_0},\partial_{a_n};\partial_{a_0})
+\partial_{a_n}\mathrm{Yuk}^{(3)}(\partial_{a_0},\partial_{a_m};\partial_{a_0})
=2\mathrm{Yuk}^{(4)}(\partial_{a_0},\partial_{a_m},\partial_{a_n};\partial_{a_0})~
\footnote{
This equation is  analogous to the case of the
compact Calabi--Yau threefold. See \cite{HKTY1}.
}
.
\end{equation}
\end{proposition}
\begin{proof} Let $\Theta_{a_m}:=a_m \mathcal{D}_{a_m}$.
\\
1. 
Notice that
for $\nabla_{\theta_{a_{m_1}}}\cdots \nabla_{\theta_{a_{m_{k-1}}}}\omega_0=
\rho(
\Theta_{a_{m_1}}\cdots \Theta_{a_{m_{k-1}}}1)
$
is expressed  in the form (cf. \eqref{reflexiveRF})
$$
\alpha_{2,0} \rho(\Theta_{a_{0}}^2 1 )+
\sum_{m\in A'(\Delta)}\alpha_{1,m} \rho(\Theta_{a_0}\Theta_{a_m}1)+
\alpha_{1,0}\rho(\Theta_{a_0}1)+ \alpha_{0,0}\rho(1)~
$$
where the coefficients satisfy
$$
\mathcal{T}_i\alpha_{2,0}=
\mathcal{T}_i\alpha_{1,0}=
\mathcal{T}_i\alpha_{1,m}=
\mathcal{T}_i\alpha_{0,0}=0~\quad (i=0,1,2).
$$
By Definition \ref{definition:Yukawa2}, we have
$$
\mathrm{Yuk}^{(k)}(\theta_{a_{m_1}},\ldots,\theta_{a_{m_{k-1}}};\theta_{a_0})
=\int_{C_a} \alpha_{2,0} \rho(\Theta_{a_{0}}^2 1 )
\wedge \rho(\Theta_{a_0}1).
$$
Then the statement follows from
Proposition \ref{prop:Ahyper}-2.
\\
The statements 
2 and 3 follow from Proposition \ref{prop:Ahyper}-2 and 
Definition \ref{definition:Yukawa2}.
\end{proof}
\subsection{Yukawa coupling for Quotient Family}
\label{sec:QuotientFamily}
Consider the action of $\T^3$ on $\mathbb{L}_{{\rm reg}}(\Delta)$:
$$
\T^3\times \mathbb{L}_{{\rm reg}}(\Delta)\to \mathbb{L}_{{\rm reg}}(\Delta)~,\quad
(\lambda_0,\lambda_1,\lambda_2)\cdot F_a(t_1,t_2)
\mapsto \lambda_0 F_a(\lambda_1 t_1,\lambda_2 t_2)~.
$$
Let $\moduli$ be the geometric invariant theory 
quotient of $\mathbb{L}_{{\rm reg}}(\Delta)$ by this action%
\footnote{Any $a\in \mathbb{L}_{{\rm reg}}(\Delta)$ is stable
in the sense of the geometric invariant theory (cf. \cite[Definition 10.5]{Batyrev}).}.
Denote the quotient map by $q:\mathbb{L}_{{\rm reg}}(\Delta)\to\moduli$.

Since $\T^3$ acts  as automorphisms on $\mathcal{Z}$,
we also have a family of affine curves
\begin{equation}\label{quotient-family}
\pi:\mathcal{Z}/\T^3\to \moduli~.
\end{equation}
(Similarly we can construct the quotient family for
the open threefold $Z^{\circ}_a$.)

The differential equation \eqref{eq:yukawa1} implies that
$\mathrm{Yuk}^{(k)}(\theta_{a_{m_1}},\ldots,\theta_{a_{m_{k-1}}};\theta_{a_0})$
depends on the parameter $a$ only through
$\T^3$-invariant combinations. Thus we can define
the Yukawa coupling for the quotient family as follows.
Let $T^0\mathcal{M}(\Delta)$ 
be the subbundle of 
the holomorphic tangent bundle $T\mathcal{M}(\Delta)$ generated by
$q_*\theta_{a_0}$.
\begin{definition}
We define a map
$$\mathrm{Yuk}^{(k)}_{\moduli}:
\underbrace{T\moduli 
\times \cdots \times
T\moduli}_{(k-1)\text{ times }}
\times T^0\moduli
\to \mathcal{O}_{\moduli}
$$
by 
$$
\mathrm{Yuk}^{(k)}_{\moduli}(A_1,\ldots, A_{k-1};A_k)=
\mathrm{Yuk}^{(k)}(A_1',\ldots,A_{k-1}';A_k'),
$$
where $A_i'$ are $\T^3$-invariant vector fields on
$\mathbb{L}_{{\rm reg}}(\Delta)$ such that
$q_*A_i'=A_i$.
The case $k=3$ is called the Yukawa coupling
and denoted by $\mathrm{Yuk}_{\moduli}$.
(We may omit the subscript $\moduli$.)
\end{definition}

In the rest of this subsection,
we rewrite the differential equations for the Yukawa coupling
(Proposition \ref{prop:Yukawa-equation1}) obtained in the previous section 
to the setting of the quotient family.
We fix a local coordinates of $\moduli$ of a particular class:
take a basis $l^{(i)}$ ($1\leq i \leq  l(\Delta)-3)$)
of the lattice of relations $L(\Delta)$.
Then $$
z_i=a^{l^{(i)}}~\qquad (1\leq i\leq l(\Delta)-3)
$$ 
form a local coordinate system on some open subset in $\moduli$.
We use the shorthand notation
\begin{equation}\nonumber
\begin{split}
&\theta_i:=\theta_{z_i}~,\quad
\theta_0:=q_*\theta_{a_0}=\sum_{i=1}^{l(\Delta)-3}
l^{(i)}_0\theta_{i}~,\quad
\nabla_i:=\nabla_{\theta_{z_i}}~,\quad
\nabla_0:=\nabla_{\theta_0}~.
\end{split}
\end{equation}
Let $\mathbf{D}$ be the set of differential operators on (some open set of)
$\moduli$, consisting of 
$$
\theta_{i_1}\cdots \theta_{i_k}\mathcal{L}_l\, , \quad
(k\geq 0,~~ 1\leq i_1,\ldots,i_k\leq l(\Delta)-3,~~ l\in L(\Delta)).
$$ 
Here $\mathcal{L}_l$ is defined by
$$
\mathcal{L}_l=q_* \Big(\prod_{m;l_m>0} a_m^{l_m}\Big) \square_l~.
$$
\begin{example}
In the case of polyhedron $\#$1 (see Example \ref{example:P2Ahyper}), 
we have the coordinate
$z=a^{(-3,1,1,1)}=\frac{a_1a_2a_3}{a_0^3}$ and
$\theta_0:=q_*\theta_{a_0}=-3\theta_z$. Then
\begin{equation}\label{eq:PF-P2}
\mathcal{L}_{(-3,1,1,1)}=\theta_z^3+3z\theta_z(3\theta_z+1)(3\theta_z+2)~,
\end{equation}
and $\mathbf{D}$ is generated by $\theta_z^k \mathcal{L}_{(-3,1,1,1)}$ ($k\geq 0$).
\end{example}

For $0\leq i_1,\ldots,i_k\leq l(\Delta)-3$,
we define 
\begin{equation}\label{def:Yij0}
Y_{i_1\ldots i_k\,;\,0}:=
\mathrm{Yuk}^{(k+1)}(\theta_{i_1},\ldots,\theta_{i_k};
\theta_{0})~.
\end{equation}

Proposition \ref{prop:Yukawa-equation1} implies the following
\begin{lemma} \label{prop:diffyukawa1}
1. 
Let $\mathcal{L}\in \mathbf{D}$
and let $U_{i_1,\ldots,i_k}\in \mathbb{C}(z)$ be
the coefficients of $\theta_{i_1}\ldots\theta_{i_k}$ in $\mathcal{L}$,
i.e.
\begin{equation}\nonumber 
\mathcal{L}=\sum_{k\geq 1}
\sum_{ i_1,\ldots,i_k} U_{i_1,\ldots,i_k}
\theta_{i_1}\cdots \theta_{i_k}~\quad 
(U_{i_1,\ldots,i_k}\in\mathbb{C}(z)).
\end{equation}
Then the Yukawa coupling satisfies
\begin{equation}\nonumber
\sum_{k\geq 2}\,\sum_{i_1,\ldots,i_k}U_{i_1\ldots i_k}  \,
Y_{i_1\ldots i_k;0}=0~.
\end{equation}
2. For $0\leq i,j\leq l(\Delta)-3$,
\begin{equation}\nonumber 
Y_{ij0;0}=\frac{1}{2}(\theta_i Y_{j0;0}+\theta_j Y_{i0;0})~.
\end{equation} 
\end{lemma}

\begin{example}\label{example:p2yukawa}
Let $\Delta$ be the polyhedron $\#$1 in Figure \ref{fig:Hirzebruch}.
Applying the above Lemma to the differential operator
\eqref{eq:PF-P2}, we obtain the equation
\begin{equation}\label{eq:p2yukawa2} 
(1+27z)\theta_z \mathrm{Yuk}(\theta_z,\theta_z;\theta_z)
+27z \mathrm{Yuk}(\theta_z,\theta_z;\theta_z)=0~,
\end{equation}
which implies
\begin{equation}\nonumber
\mathrm{Yuk}(\theta_z,\theta_z;\theta_z)=-\frac{c}{3(1+27z)}
\end{equation}
where $c$ is some nonzero constant.
This result is the same as 
Example \ref{p2yukawa1}.

\begin{remark}\label{p2wronskian}
Let $t,\partial_S F$ be the solutions \eqref{eq:PFsolution}
of the $A$-hypergeometric 
system  associated to the 
polyhedron $\#$1 in Figure \ref{fig:Hirzebruch}. 
Then we have
\begin{equation}
\mathrm{Yuk}(\partial_t,\partial_t;\partial_t)\propto
\partial_t^2 \partial_S F~.
\end{equation}
This follows from  the multilinearity of $\mathrm{Yuk}$ and the fact that
$$
\wronskian(t,\partial_S F):=\det\begin{pmatrix}
\theta_z^2 t &\theta_z t\\
\theta_z^2\partial_S F &\theta_z \partial_S F
\end{pmatrix}
=-(\theta_z t)^3\cdot \partial_t^2\partial_S F
$$
is proportional to $\mathrm{Yuk}(\theta_z,\theta_z;\theta_z)$
since it
satisfies the same differential equation \eqref{eq:p2yukawa2}. 
\end{remark}
\end{example}
\subsection{Comments on Yukawa coupling in the local A-model and local mirror symmetry}
Let $\Delta$ be a $2$-dimensional reflexive polyhedron.
Consider the $2$-dimensional nonsingular complete 
smooth fan $\Sigma(\Delta^*)$
whose generators of 1-cones are $A(\Delta)\setminus \{0\}$. 
Let $\mathbb{P}_{\Sigma(\Delta^*)}$ 
be the toric surface defined by $\Sigma(\Delta^*)$.
For example, $\mathbb{P}_{\Sigma(\Delta^*)}=\mathbb{P}^2$ if $\Delta$ is the 
polyhedron $\#$1 in Figure \ref{fig:Hirzebruch}.
Take a basis $C_i$ $(1\leq i\leq l(\Delta)-3)$ of 
$H_2(\mathbb{P}_{\Sigma(\Delta^*)},\mathbb{Z})
\cong L(\Delta)$ and
let $J_i$  $(1\leq i\leq l(\Delta)-3)$ be
the dual basis.
Denote by
$t_i$ $(1\leq i\leq l(\Delta)-3)$ the coordinates on 
$H^2(\mathbb{P}_{\Sigma(\Delta^*)})$
associated to this basis.
Let $c_i$ be the coefficients of $J_i$ in $c_1(\mathbb{P}_{\Sigma(\Delta^*)})=
\sum_{i}c_i J_i$ and let $J_i\cdot J_j$ be the intersection numbers.

Let $N_{0,\beta}(\mathbb{P}_{\Sigma(\Delta^*)})$ be the genus zero
local Gromov--Witten invariant of degree $\beta$,
and define $F_{\mathrm{inst}}^{\mathbb{P}_{\Sigma(\Delta^*)}}(t)$ by
$$
F_{\mathrm{inst}}^{\mathbb{P}_{\Sigma(\Delta^*)}}(t)=
\sum_{\beta=\sum d_i C_i}
N_{\beta}^{\mathbb{P}_{\Sigma(\Delta^*)}} e^{\sum d_it_i}~.
$$
Note that $\dim H^2(\mathbb{P}_{\Sigma(\Delta^*)})=\dim 
{R}_{F_a}^1=l(\Delta)-3$.
Let $t_i(z)$ be solutions 
of the $A$-hypergeometric system with a single logarithm,
so-called the mirror maps,
and let $\partial_SF$ be a solution with double logarithms. 
(See \cite[eq.(6.22)]{CKYZ} for definitions of 
$t_i,\partial_S F$.
$\Pi_i$ there is $t_i$ here.)
Local mirror symmetry \cite{CKYZ} says that, 
under 
an appropriate identification between $t_i$'s and $t_i(z)$'s,
$\partial_SF$ is 
related to the local Gromov--Witten invariants by
$$
\partial_S F=\sum_{i,j=1}^{l(\Delta)-3}\frac{J_i\cdot J_j}{2}t_it_j
-\sum_{i=1}^{l(\Delta)-3}c_i \partial_{t_i}  
F_{\mathrm{inst}}^{\mathbb{P}_{\Sigma(\Delta^*)}}(t)~.
$$

Let $T^0H^2(\mathbb{P}_{\Sigma(\Delta^*)})$ be 
the one-dimensional subspace of $TH^2(\mathbb{P}_{\Sigma(\Delta^*)})$
spanned by $\sum_{i}c_i\partial_{t_i}$.
The local A-model Yukawa coupling  $\mathrm{Yuk}_A$
may be defined as 
a multilinear map from
$
TH^2(\mathbb{P}_{\Sigma(\Delta^*)})
\times TH^2(\mathbb{P}_{\Sigma(\Delta^*)})
\times T^0H^2(\mathbb{P}_{\Sigma(\Delta^*)})
$ 
to
$\mathcal{O}_{ H^2(\mathbb{P}_{\Sigma(\Delta^*)}) }$
given by
\begin{equation}\nonumber
\mathrm{Yuk}_A\Big(\partial_{t_i},\partial_{t_j};
\sum_{l=1}^{l(\Delta)-3}c_l \partial_{t_l}
\Big)
=\partial_{t_i}\partial_{t_j}\partial_SF~.
\end{equation}

\begin{example}\label{P2Ayukawa}
Let $\Delta$ be the polyhedron $\#$1 in Figure \ref{fig:Hirzebruch}.
As in
Remark \ref{p2wronskian},
the local A-model Yukawa coupling $\mathrm{Yuk}_A$ is proportional to
the local B-model Yukawa coupling $\mathrm{Yuk}$.
To get the equality, we set $c=1$ 
in Example \ref{example:p2yukawa}. 
\end{example}

We also see that for the other polyhedra in Figure \ref{fig:Hirzebruch}, 
the Yukawa couplings
coincide with the local A-model Yukawa couplings $\mathrm{Yuk}_A$
under the mirror maps $t_1,\, t_2$. See \S \ref{section:Examples}.
\section{Holomorphic anomaly equation}\label{section:HAE}
\subsection{Analogue of Special K\"ahler Geometry}
We propose an analogue of the special geometry relation
for $\moduli$. 
Consider the quotient family $\pi:\mathcal{Z}/\T^3\to \moduli$.
We use the same notations
$z_i,\theta_0,\theta_i,\nabla_0,\nabla_i$ 
as in \S \ref{sec:QuotientFamily}.
Let 
\begin{equation}\label{def:phi}
\phi:=\nabla_{0}\omega_0 \in H^1(C_z)~.
\end{equation}

As in \eqref{def:Yij0}, we set 
$$
Y_{i\,0;0}=\sqrt{-1}\int_{C_z}\nabla_i \phi\wedge \phi
\quad
(0\leq i\leq l(\Delta)-3)~.
$$
We also set
$$
G_{0\overline{0}}:=-\sqrt{-1}\int_{C_z} \phi\wedge \overline{\phi}~.
$$
This defines a Hermitian metric on 
$T^0\mathcal{M}(\Delta)$ such that the norm of $\theta_0$
is $G_{0\overline{0}}$.

By the definition of $G_{0\overline{0}}$, 
$Y_{00;0}$ and $Y_{i\,0;0}$, we have the following
\begin{lemma}\label{prop:specialgeometry}
\begin{equation}
\begin{split}\nonumber
(1) \quad &
\nabla_i\phi=\frac{\theta_i G_{0\overline{0}}}{G_{0\overline{0}}}
\phi
+\frac{Y_{i0;0}}{G_{0\overline{0}}}\overline{\phi}~,\\
(2)\quad & 
\overline{\theta}_j \frac{\theta_iG_{0\overline{0}}}{G_{0\overline{0}}}=
-\frac{Y_{i0;0}\overline{Y}_{j0;0}}{G_{0\overline{0}}^2}~.
\end{split}
\end{equation}
Let
\begin{equation}\nonumber
\kappa:=\theta_0\frac{\theta_0G_{0\overline{0}}}{G_{0\overline{0}}}+
  \Big(\frac{\theta_0 G_{0\overline{0}}}{G_{0\overline{0}}}\Big)^2 
  -\frac{\theta_0 Y_{00;0}}{Y_{00;0}}
\frac{\theta_{0}G_{0\overline{0}}}{G_{0\overline{0}}}~.
\end{equation}
Then
\begin{equation}\begin{split}\nonumber
(3)\quad &\overline{\theta}_j\kappa =0  \qquad (1\leq j\leq l(\Delta)-3)~,
\\
(4) \quad & 
\nabla_0^2 \phi=\kappa \phi+\frac{\theta_0 Y_{00;0}}{Y_{00;0}}\nabla_0 \phi~.
\end{split}
\end{equation}
\end{lemma}
The second equation is analogous to the special geometry equation \cite{BCOV}.
The third equation is an analogue of \cite[eq.(3.2)]{YamaguchiYau}.

\begin{example}
Let $\Delta$ be the polyhedron $\#$1 in Figure \ref{fig:Hirzebruch}.
By comparing the fourth equation of the above lemma and
the differential operator \eqref{eq:PF-P2},
we have $$
Y_{00;0}=\frac{9}{1+27z}~,\qquad
\kappa= -\frac{54z}{1+27z}\, .
$$
\end{example}
\subsection{Proposal of local holomorphic anomaly equation}
We propose how to adapt BCOV's 
holomorphic anomaly equation \cite{BCOV} to the local B-model.
Let $\tilde{C}_n^g$ ($g,n\geq 0$)
be the $n$-point  B-model topological string amplitude of genus $g$.
For $2g-2+n\leq 0$, we set
\begin{equation}
\tilde{C}_0^0=\tilde{C}_1^0=\tilde{C}_2^0=0~,\quad
\tilde{C}_0^1=0~.
\end{equation}
For $2g-2+n\geq 1$, we put
$$
\tilde{C}_{n+1}^g=\Big(\theta_0-n\frac{\theta_0 G_{0\overline{0}}}{G_{0\overline{0}}}\Big)
\tilde{C}_n^g~.
$$
For $(g,n)=(0,3)$, let
\begin{equation}\label{eq:HAE0}
\tilde{C}_3^0=Y_{00;0}~.
\end{equation}
As a holomorphic anomaly equation for
 $(g,n)=(1,1)$, we propose
\begin{equation}\label{eq:HAE1}
\overline{\theta}_j\tilde{C}_1^1=-\frac{1}{2}\overline{\theta}_j 
\frac{\theta_0 G_{0\overline{0}}}{G_{0\overline{0}}}~,
\quad\text{  which implies that  }\quad
\tilde{C}_1^1=-\frac{1}{2}\frac{\theta_0 G_{\overline{0}}}{G_{0\overline{0}}}
+f^1_1(z)~.
\end{equation}
For $(g,n)=(g,0)$ $(g\geq 2)$, we propose
\begin{equation}\label{eq:HAE2}
\overline{\theta}_j \tilde{C}_0^g=
\frac{\overline{Y}_{j0;0}}{2 G_{0\overline{0}}^2}
\big( \tilde{C}_2^{g-1}+\sum_{h_1+h_2=g} \tilde{C}_1^{h_1}\tilde{C}_1^{h_2}
\big)~.
\end{equation}

For $g\geq 2$, $\tilde{C}_0^g$ can be solved by
the Feynman diagram method as in \cite{BCOV} 
or Yamaguchi--Yau's polynomial method as in \cite{YamaguchiYau}.
\subsection{Solution by Feynman diagram \cite{BCOV}}
Define the propagator  $S^{00}$ by the differential equation
$\overline{\theta}_j S^{00}=\frac{\overline{Y}_{j0;0}}{G_{0\overline{0}}^2}
$.
It is easily solved by Lemma \ref{prop:specialgeometry}-(2):
$$
S^{00}=-\frac{1}{Y_{00;0}}\frac{\theta_0 G_{0\overline{0}}}{G_{0\overline{0}}}+
f_s(z)~,
$$
where $f_s(z)$ is a meromorphic function in $z$.
Put $\Delta_{00}:=-1/S^{00}$.
Then assuming \eqref{eq:HAE0}, \eqref{eq:HAE1} and \eqref{eq:HAE2}, we can show that
$$
\overline{\theta}_j
\exp\Big[
-\frac{1}{2\lambda^2}\Delta_{00}x^2+\frac{1}{2}\log \frac{\Delta_{00}}{\lambda^2}
+\sum_{n,g\geq 0}\frac{\lambda^{2g-2}}{n!}\tilde{C}_n^g x^n
\Big]=0~.
$$
This implies that $\tilde{C}_0^g$ ($g\geq 2$) can be 
computed as a sum over Feynman diagrams of genus $g$.
The difference from the one given in \cite{BCOV} is that 
there is only one propagator, $S^{00}$.
\subsection{Solution by Yamaguchi--Yau's method \cite{YamaguchiYau}}
Let 
$$
A=\frac{\theta_0 G_{0\overline{0}}}{G_{0\overline{0}}}~.
$$
From the above Feynmann diagram method and
the fact that $\theta_0 A\in \mathbb{C}(z)[A]$
(see Lemma \ref{prop:specialgeometry}-(3)), 
it follows that
$\tilde{C}_n^g$ is a polynomial of degree $3g-3+n$ in $\mathbb{C}(z)[A]$.
Moreover, it satisfies
\begin{equation}
\frac{\partial \tilde{C}_0^g}{\partial A}=
-\frac{1}{2Y_{00;0}}\Big(
\tilde{C}_2^{g-1}+\sum_{h_1+h_2=g}\tilde{C}_1^{h_1}\tilde{C}_1^{h_2}
\Big)~.
\end{equation}

\begin{example}
For $(g,n)=(1,1),(1,2)$ and $(2,0)$, we have
\begin{equation}
\begin{split}
\tilde{C}_1^1&=-\frac{1}{2}A+f_1^1(z)~,\quad
\tilde{C}_2^1=A^2+A\Big(-\frac{\theta_0 Y_{00;0}}{2Y_{00;0}}-f_1^1\Big)
-\frac{\kappa}{2}+\theta_0 f_1^1~,
\\
\tilde{C}_0^2&=-\frac{1}{2Y_{00;0}}
\Big[
\frac{5}{12}A^3-\Big(\frac{\theta_0 Y_{00;0}}{4Y_{00;0}}+f_1^1\Big)A^2
+\Big(-\frac{\kappa}{2}+\theta_0 f_1^1+(f_1^1)^2\Big)A
\Big]+f_2(z)~.
\end{split}
\end{equation}
\end{example}

\begin{example}
Let $\Delta$ be the polyhedron $\#$1 in Figure \ref{fig:Hirzebruch}.
We checked that $\tilde{C}_1^1,\tilde{C}_0^2$ 
give the correct local GW invariants of $\mathbb{P}^2$ at least in small degrees.
The holomorphic ambiguities are
\begin{equation}\nonumber
\begin{split}
f_1^1(z)&=\frac{1+54z}{4(1+27z)}~,\qquad
f_2(z)=\frac{\frac{3}{40}z+\frac{783}{80}z^2+\frac{3645}{8}z^3}{(1+27z)^2}~.
\end{split}
\end{equation}
The holomorphic limit is
$$
G_{0\overline{0}}\rightarrow \theta_z t~.
$$
\end{example}

\subsection{Witten's geometric quantization approach}
First recall Witten's geometric quantization and its
implication for holomorphic anomaly equation \cite{Witten}.
Let $W=\mathbb{R}^{2N}$ be a vector space 
equipped with the standard symplectic form and
let $L\to W$ be a complex line bundle whose connection $1$-form
is the canonical $1$-form.
Let
$\mathcal{M}$ be the space of complex structures on $W$.
To each complex structure $J\in \mathcal{M}$,
associate the holomorphic polarization $\mathcal{H}_J$
which is a subspace of the space of square integrable 
sections $\Gamma(W, L)$ consisting of 
``holomorphic'' ones.
Then an infinite dimensional bundle $\mathcal{H}\to \mathcal{M}$ is obtained.
Witten found a projectively flat connection on $\mathcal{H}$.
His claim is that if this is applied to
the case where $W=H^3(X^{\vee},\mathbb{R})$ is the cohomology of a
Calabi--Yau threefold $X^{\vee}$,
then BCOV's holomorphic anomaly equation appears as 
the condition for the flatness of a section of $\mathcal{H}$.

We apply Witten's idea to the case when
$W=\mathcal{W}_1 H^1(C_z^{\circ},\mathbb{R})=H^1(C_z,\mathbb{R})$
and $\mathcal{M}=\mathcal{M}(\Delta)$.
(To be precise, $\mathcal{M}(\Delta)$ is not the space of complex structures 
of $W$ but it is larger in general. However, this point does not matter in the following 
argument.)
Take $\phi,\overline{\phi}$ defined in \eqref{def:phi}
as a basis of 
$W_{\mathbb{C}}=\mathcal{W}_1H^1(C^{\circ}_z)=H^1(C_z)$
and let $x,\overline{x}$ be the associated complex coordinates.
$W$ has a symplectic form $\sqrt{-1}G_{0\overline{0}}dx\wedge d\overline{x}$ 
given by the intersection product.
Consider the  
trivial line bundle 
$L=\mathbb{C}\times W$ with the connection
$$
\delta+\frac{1}{2} G_{0\overline{0}}(xd\overline{x}-\overline{x}dx)~.
$$
Here we use $\delta$ to denote the differential on $W$.
Then the holomorphic polarization $\mathcal{H}_z$ ($z\in \mathcal{M}(\Delta)$)
is as follows:
\begin{equation}\nonumber
\begin{split}
\mathcal{H}_z&=\Big\{\Phi\in \Gamma(W,L)\mid 
\Big(\overline{\delta}_{\overline{x}}+\frac{G_{0\overline{0}}}{2}x\Big)\Phi=0
\Big\}
\\
&=\big\{\Phi\in \Gamma(W,L)\mid \Phi=\varphi(x) 
e^{-\frac{G_{0\overline{0}}}{2}x\overline{x}}~ \big\}~.
\end{split}
\end{equation}
Mimicking Witten's result, we can show that 
\begin{equation}\nonumber
\begin{split}
\theta_j \mathcal{H}\subset \mathcal{H}~,\qquad
&\Bigg(\overline{\theta}_j-\frac{\overline{Y}_{j0;0}}{2G_{0\overline{0}}^2}
\Big(\delta_x-\frac{G_{0\overline{0}}}{2}\overline{x}\Big)^2
\Bigg)\mathcal{H}\subset \mathcal{H}~.
\end{split}
\end{equation}
Moreover these make a projectively flat connection on $\mathcal{H}$.

If  we regard 
$$
\exp\Big[\sum_{n,g\geq 0} \frac{\lambda^{2g-2+n}}{n!}\tilde{C}_n^g x^n\Big]
\times e^{-\frac{G_{0\overline{0}}}{2}x \overline{x}}
$$
as a section of $\mathcal{H}$, then the
condition that it is a flat section
results in the following equation:
\begin{equation}\nonumber
\overline{\theta}_j \tilde{C}_n^g
=\frac{\overline{Y}_{j0;0}}{2G_{0\overline{0}}^2}
\Bigg(
\tilde{C}_{n+2}^{g-1}+\sum_{\begin{subarray}{c}h_1+h_2=g,\\0\leq m\leq n
\end{subarray}}\begin{pmatrix}n\\m\end{pmatrix}
\tilde{C}_{m+1}^{h_1}\tilde{C}_{n-m+1}^{h_2}
\Bigg)~.
\end{equation}
\section{Examples}\label{section:Examples}
In this section, we consider the 
polyhedra \#2, 3, 4 in Figure \ref{fig:Hirzebruch}.
\subsection{$\mathbb{F}_0$ case}
Let $\Delta$ be the polyhedron $\#$2 in Figure \ref{fig:Hirzebruch}:
\begin{equation}\nonumber
\Delta=\text{~
the convex hull of~} \{(1,0),(0,1),(-1,0),(0,-1)\}~.
\end{equation}
\subsubsection*{{\bf $\Delta$-regularity condition}}
The $\Delta$-regularity condition for 
$F\in \mathbb{L}(\Delta)$
is as follows:
\begin{equation}
\begin{split}
&F(t_1,t_2)=a_0+a_1t_1+a_2t_2+\frac{a_3}{t_1}+\frac{a_4}{t_2}~,
\\
&a_1a_2a_3a_4\neq 0~,\quad
(a_0^2-4a_1a_3-4a_2a_4)^2-64a_1a_2a_3a_4\neq 0~.
\end{split}
\end{equation}
\subsubsection*{{\bf $\mathcal{R}_F$ and filtrations}}
We have
\begin{equation}\nonumber
\mathcal{R}_F\cong \mathbb{C}\,1\oplus\mathbb{C}\,t_0\oplus
            \mathbb{C}\,t_0t_1\oplus\mathbb{C}\,t_0^2~.
\end{equation}
The $\mathcal{I}$-filtration 
and the $\mathcal{E}$-filtration are as follows.
\begin{equation}\nonumber
\mathcal{I}_1=\mathcal{I}_2=\mathbb{C}\,t_0\oplus\mathbb{C}\,t_0^2~,\quad
\mathcal{I}_3=\mathcal{I}_1\oplus \mathbb{C}\,t_0t_1~,\quad
\quad\mathcal{I}_4=\mathcal{R}_F~.
\end{equation}
\begin{equation}\nonumber
\mathcal{E}^0=\mathbb{C}\,1~,\quad
\mathcal{E}^{-1}=\mathcal{E}^0\oplus \mathbb{C}\,t_0\oplus\mathbb{C}\,t_0t_1~,
\quad
\mathcal{E}^{-2}=\mathcal{R}_F~.
\end{equation}
\subsubsection*{{\bf MHS}}
By Theorem \ref{theorem:MHS} and \eqref{reflexiveRF},
$$
H^2(\T^2,C^{\circ}_a)=\mathbb{C}\omega_0\oplus PH^1(C^{\circ}_a)~,\quad
PH^1(C^{\circ}_a)=\mathbb{C}\,\rho (t_0)\oplus
            \mathbb{C}\rho(t_0t_1)\oplus\mathbb{C}\rho(t_0^2)~.
$$
\begin{equation}\nonumber
\mathcal{W}_1=\mathbb{C}\rho(t_0)\oplus\mathbb{C}\rho(t_0^2)~,\quad
\mathcal{W}_2=\mathcal{W}_1\oplus \mathbb{C}\rho(t_0t_1)~,\quad
\quad\mathcal{W}_4=H^2(\T^2,\mathbb{C})~.
\end{equation}
\begin{equation}\nonumber
\mathcal{E}^0=\mathbb{C}\omega_0~,\quad
\mathcal{E}^{-1}=\mathcal{E}^0\oplus \mathbb{C}\rho(t_0)
    \oplus\mathbb{C}\rho(t_0t_1)~,
\quad
\mathcal{E}^{-2}=H^2(\T^2,\mathbb{C})~.
\end{equation}
\subsubsection*{{\bf $A$-hypergeometric system}}
The lattice of relations $L(\Delta)$ (defined in \eqref{eq:lattice-relations})
is generated by two vectors
\begin{equation}\nonumber
l^{(1)}=(-2,1,0,1,0)~,\qquad l^{(2)}=(-2,0,1,0,1)~.
\end{equation}
The $A$-hypergeometric system 
is generated by the following differential operators:
\begin{equation}\nonumber 
\begin{split}
&\theta_{a_1}-\theta_{a_3}~,\quad
\theta_{a_2}-\theta_{a_4}~,\quad
\theta_{a_1}+\theta_{a_2}+\theta_{a_3}+\theta_{a_4}+\theta_{a_0}~,
\\
&\partial_{a_1}\partial_{a_3}-\partial_{a_0}^2~,
\quad
\partial_{a_2}\partial_{a_4}-\partial_{a_0}^2~.
\end{split}
\end{equation}

Take 
$$ z_1=a^{l^{(1)}}=\frac{a_{1}a_{3}}{a_0^2}~,\qquad
z_2=a^{l^{(2)}}=\frac{a_{2}a_{4}}{a_0^2}~. 
$$
These are coordinates of an open subset of $\moduli$.
We have
$\theta_0:=q_*\theta_{a_0}=-2\theta_{z_1}-2\theta_{z_2}$. 
With these coordinates, 
the above $A$-hypergeometric system reduces to the following
two differential operators of order $2$:
\begin{equation}\begin{split}\nonumber 
&\mathcal{L}_1=\theta_1^2-z_1(-2\theta_1-2\theta_2)
(-2\theta_1-2\theta_2-1)~,
\\
&\mathcal{L}_2=\theta_2^2-z_2(-2\theta_1-2\theta_2)
(-2\theta_1-2\theta_2-1)~.
\end{split}
\end{equation}
Solutions about $z_1=0,z_2=0$ are as follows.
\begin{equation}
\begin{split} \nonumber 
&\varpi(z;0)=1~,\\
t_1:=\partial_{\rho_1}&\varpi(z;\rho)|_{\rho=0}=\log z_1+2H(z_1,z_2)~,
\\
t_2:=\partial_{\rho_2}&\varpi(z;\rho)|_{\rho=0}=\log z_2+2H(z_1,z_2)~,
\\
\partial_S F:=\partial_{\rho_1}\partial_{\rho_2}
 &\varpi(z;\rho)=\log z_1\log z_2+\cdots,
\end{split}
\end{equation}
where
\begin{equation}\nonumber
\begin{split}
\varpi(z;\rho)&=\sum_{n_1,n_2\geq 0}
\frac{(2\rho_1+2\rho_2)_{2n_1+2n_2}}{(\rho_1+1)_{n_1}^2(\rho_2+1)_{n_2}^2}
z_1^{n_1+\rho_1}z_2^{n_2+\rho_2}~,
\\
H(z_1,z_2)&=\sum_{\begin{subarray}{c}n_1,n_2\geq 0\\
(n_1,n_2)\neq (0,0)\end{subarray}}
\frac{(2n_1+2n_2-1)!}{n_1!^2n_2!^2}z_1^{n_1}z_2^{n_2}~.
\end{split}
\end{equation}
\subsubsection*{{\bf Yukawa coupling}}
In this case, $\mathbf{D}$ is generated by $\mathcal{L}_1,\mathcal{L}_2$.
Applying Lemma \ref{prop:diffyukawa1} 
to $\mathcal{L}_1$, $\mathcal{L}_2$,
$\theta_0 \mathcal{L}_1$, $\theta_0\mathcal{L}_2~$,
we obtain first order partial differential equations for $Y_{i,j:0}$. 
Solving these equations, we obtain: 
\begin{equation}\begin{split}\nonumber
&Y_{0,0;0}=\frac{8c}{d(z_1,z_2)}~,
\\
&Y_{1,1;0}=\frac{8 c z_1}{d(z_1,z_2)}~,\quad
Y_{1,2;0}=\frac{c(1-4z_1-4z_2)}{d(z_1,z_2)}~,\quad
Y_{2,2;0}=\frac{8 cz_2}{d(z_1,z_2)}~,
\end{split}\end{equation} 
where $d(z_1,z_2)=(1-4z_1-4z_2)^2-64z_1z_2$
and $c\in \mathbb{C}$ is a nonzero constant.

\subsubsection*{{\bf Comparison with the local  A-model Yukawa coupling}}
We show that the Yukawa coupling  and 
the local A-model Yukawa coupling coincide under the mirror map:

\begin{equation}\label{F0yukawa2}
\mathrm{Yuk}(\partial_{t_{\alpha}},\partial_{t_{\beta}};
-2\partial_{t_{1}}-2\partial_{t_{2}})
\propto 
\partial_{t_{\alpha}}\partial_{t_{\beta}}\partial_SF~.
\end{equation}
For this purpose, let us define the ``Wronskian'' of $t_1,t_2,\partial_S F$
by
\begin{equation}\label{eq:defwronskian}
\begin{split}
\wronskian_{i_1\ldots i_k}(t_1,t_2,\partial_S F)
&:=\det
\begin{pmatrix}
\theta_{i_1}\cdots\theta_{{i_k}}t_1&
\theta_{1}t_1&\theta_{2}t_1\\
\theta_{i_1}\cdots\theta_{{i_k}}t_2&
\theta_{1}t_2&\theta_{2}t_2\\
\theta_{i_1}\cdots\theta_{{i_k}}\partial_SF&
\theta_{1}\partial_SF&\theta_{2}\partial_SF
\end{pmatrix}
\\
&=
\det\begin{pmatrix}
\theta_1t_1&\theta_2t_1\\
\theta_1t_2&\theta_2t_2
\end{pmatrix}\cdot
\sum_{\alpha,\beta=1}^2
\partial_1t_{\alpha}\cdot\partial_2 t_{\beta} \cdot 
\partial_{t_{\alpha}}\partial_{t_{\beta}}\partial_S F
~~.
\end{split}
\end{equation}
We can  show that  
Lemma \ref{prop:diffyukawa1}  
holds if we replace $\mathrm{Yuk}_{i_1,\ldots,i_k;0}$ with
$\wronskian_{i_1\ldots i_k}(t_1,t_2,\partial_S F)$
\footnote{
The first statement follows from
the cofactor expansion of the determinant and 
the fact that $t_1,t_2,\partial_S F$
are solutions of $\mathcal{L}=0$ for $\mathcal{L}\in \mathbf{D}$:
$$
\sum_{i_1,\ldots,i_k} 
U_{i_1\ldots i_k} \mathrm{Wr}_{i_1\ldots i_k}(t_1,t_2,\partial_S F)=
\det\begin{pmatrix}\theta_1 t_2&\theta_2 t_2\\
\theta_1 \partial_S F&\theta_2\partial_S F\end{pmatrix}\mathcal{L}t_1
-\det\begin{pmatrix}
\theta_1 t_1&\theta_2 t_1\\
\theta_1 \partial_S F&\theta_2\partial_S F
\end{pmatrix}\mathcal{L}t_2
+\det \begin{pmatrix}
\theta_1 t_1&\theta_2 t_1\\
\theta_1 t_2&\theta_2 t_2
\end{pmatrix}\mathcal{L}\partial_S F
=0~.
$$
To prove the second statement, we first 
solve $\mathcal{L}_1*=\mathcal{L}_2*=0$
and express $\theta_1^2*,\theta_2^2*$ in terms of 
$\theta_1\theta_2*,\theta_1*,\theta_2*$ 
($*=t_1,t_2,\partial_S F$).
Then if we substitute these into 
$\theta_1 \mathrm{Wr}_{11}(t_1,t_2,\partial_S F)
-\mathrm{Wr}_{110}(t_1,t_2,\partial_S F)$,
terms cancell each other and we obtain zero.
We can prove the other equations similarly.
}%
.
Therefore $\wronskian_{ij}(t_1,t_2,\partial_S F)$
must be proportional to $Y_{ij;0}$.
Then \eqref{F0yukawa2} follows from the multilinearity of $\mathrm{Yuk}$.
\subsubsection*{{\bf Holomorphic ambiguities}}
The multiplication constant of $Y_{00;0}$ is $c=1$.
From $\mathcal{L}_1,\mathcal{L}_2$, we obtain
$$
\kappa=\frac{8(z_1+z_2-6(z_1^2+z_2^2)+12z_1z_2)}{d(z_1,z_2)}~.
$$
We checked that $\tilde{C}_1^1,\tilde{C}_0^2$ 
give the correct local GW invariants of $\mathbb{F}_0$
for small degrees.
The holomorphic ambiguities are
\begin{equation}\nonumber
\begin{split}
f_1^1(z)&=-\frac{1}{12}\frac{\theta_0 d(z_1,z_2)}{d(z_1,z_2)}+\frac{1}{6}~,
\quad
f_2(z)=\frac{1}{d(z_1,z_2)^2}\Big(\sum_{i,j=0}^5 b_{ij} z_1^iz_2^j\Big)~.
\end{split}
\end{equation}
(The  numerator of $f_2(z)$ is omitted because it is long.)
As the holomorphic limit, we take 
$$
G_{0\overline{0}}\rightarrow 1-\theta_0 H(z_1,z_2).
$$
\subsection{$\mathbb{F}_1$ case}
Let $\Delta$ be the polyhedron $\#$3 in Figure \ref{fig:Hirzebruch}:
\begin{equation}\nonumber
\Delta=\text{~
the convex hull of~} \{(1,0),(0,1),(-1,0),(-1,-1)\}~.
\end{equation}
\subsubsection*{{\bf $\Delta$-regularity}}
The $\Delta$-regularity condition for 
$F\in \mathbb{L}(\Delta)$
is as follows:
\begin{equation}\nonumber
\begin{split}
&F(t_1,t_2)=a_0+a_1t_1+a_2t_2+\frac{a_3}{t_1}+\frac{a_4}{t_1t_2}~,
\\
&a_1a_2a_3a_4\neq 0~,\quad
a_3(a_0^2-4a_1a_3)^2-a_2a_4(a_0^3-36a_0a_1a_3+27a_1a_2a_4)\neq 0~.
\end{split}
\end{equation}
\subsubsection*{{\bf $\mathcal{R}_F$, $\mathcal{I}$-filtration, $\mathcal{E}$-filtration
and MHS}} 
These are the same as the $\mathbb{F}_0$-case.
\subsubsection*{{\bf $A$-hypergeometric system}}
The lattice of relations $L(\Delta)$
is generated by two vectors
\begin{equation}\nonumber
l^{(1)}=(-2,1,0,1,0)~,\qquad l^{(2)}=(-1,0,1,-1,1)~.
\end{equation}
The $A$-hypergeometric system 
is generated by the following differential operators:
\begin{equation}\nonumber 
\begin{split}
&\theta_{a_1}-\theta_{a_3}-\theta_{a_4}~,\quad
\theta_{a_2}-\theta_{a_4}~,\quad
\theta_{a_1}+\theta_{a_2}+\theta_{a_3}+\theta_{a_4}+\theta_{a_0}~,
\\
&\partial_{a_1}\partial_{a_3}-\partial_{a_0}^2~,
\quad
\partial_{a_2}\partial_{a_4}-\partial_{a_0}\partial_{a_3}~.
\end{split}
\end{equation}

Take 
$$ z_1=a^{l^{(1)}}=\frac{a_{1}a_{3}}{a_0^2}~,\qquad
z_2=a^{l^{(2)}}=\frac{a_{2}a_{4}}{a_0a_3}~. 
$$
These are coordinates of an open subset of $\moduli$.
We have
$\theta_0:=
q_*\theta_{a_0}=-2\theta_{z_1}-\theta_{z_2}$. 

With these coordinates, 
the $A$-hypergeometric system reduces to the following 
two differential operators of order $2$:
\begin{equation}\begin{split} \nonumber 
&\mathcal{L}_1=\theta_1(\theta_1-\theta_2)-z_1(-2\theta_1-\theta_2)
(-2\theta_1-\theta_2-1)~,
\\
&\mathcal{L}_2=\theta_2^2-z_2(-2\theta_1-\theta_2)(\theta_1-\theta_2)~.
\end{split}
\end{equation}
Solutions about $z_1=0,z_2=0$ are as follows.
\begin{equation}
\begin{split}\nonumber 
&\varpi(z;0)=1~,\\
t_1:=\partial_{\rho_1}&\varpi(z;\rho)|_{\rho=0}=\log z_1+2H(z_1,z_2)~,
\\
t_2:=\partial_{\rho_2}&\varpi(z;\rho)|_{\rho=0}=\log z_2+H(z_1,z_2)~,
\\
\partial_S F:=\Big(\frac{1}{2}&\partial_{\rho_1}^2+\partial_{\rho_1}\partial_{\rho_2}\Big)
 \varpi(z;\rho)~,
\end{split}
\end{equation}
where
\begin{equation}\nonumber
\begin{split}
\varpi(z;\rho)&=\sum_{n_1,n_2\geq 0}
\frac{(2\rho_1+\rho_2)_{2n_1+n_2}}{(\rho_1+1)_{n_1}(\rho_2+1)_{n_2}^2}
\frac{\Gamma(1+\rho_1-\rho_2)}{\Gamma(1+\rho_1-\rho_2+n_1-n_2)}
z_1^{n_1+\rho_1}z_2^{n_2+\rho_2}~,
\\
H(z_1,z_2)&=\sum_{\begin{subarray}{c}n_1,n_2\geq 0\\
n_1\geq n_2\end{subarray}}
\frac{(2n_1+n_2-1)!}{n_1!(n_1-n_2)! n_2!^2}(-1)^{n_2}z_1^{n_1}z_2^{n_2}~.
\end{split}
\end{equation}
Here $\Gamma(x)$ denotes the Gamma function.
\subsubsection*{{\bf Yukawa coupling}}
\begin{equation}\begin{split}
&Y_{0,0;0}=\frac{c(8-9z_2)}{d(z_1,z_2)}
~,\quad
Y_{1,1;0}=\frac{c(1+4z_1-z_2-3z_1z_2)}{d(z_1,z_2)}~,
\\&
Y_{1,2;0}=\frac{c(1-4z_1-z_2+6z_1z_2)}{d(z_1,z_2)}~,
\quad
Y_{2,2;0}=-\frac{c(z_2(1+12z_1))}{d(z_1,z_2)}~,
\end{split}\end{equation}
where $d(z_1,z_2)=(1-4z_1)^2-z_2(1-36z_1+27z_1z_2)$
and $c\in \mathbb{C}$ is a nonzero constant.
\subsubsection*{{\bf Comparison with local A-model Yukawa coupling}}
As in the $\mathbb{F}_0$-case, we can show that
$$
\mathrm{Yuk}
(\partial_{t_{\alpha}},\partial_{t_{\beta}};-2\partial_{t_{1}}-\partial_{t_2})
\propto \partial_{t_{\alpha}}\partial_{t_{\beta}}\partial_S F~
\quad (1\leq \alpha,\beta\leq 2)~.
$$
\subsubsection*{{\bf Holomorphic ambiguities}}
The multiplication constant of $Y_{00;0}$ is $c=1$ and
$$
\kappa=\frac{2z_1(-32+192z_1+282z_2-144 z_1z_2-486 z_2^2+243 z_2^3)}{d(z_1,z_2)}~.
$$
We checked that $\tilde{C}_1^1,\tilde{C}_0^2$ 
give the correct  local GW invariants of $\mathbb{F}_1$ for small degrees.
The holomorphic ambiguities are
\begin{equation}\nonumber
\begin{split}
f_1^1(z)&=-\frac{1}{12}\frac{\theta_0 d(z_1,z_2)}{d(z_1,z_2)}+\frac{1}{6}~,
\quad
f_2(z)=\frac{1}{d(z_1,z_2)^2}\Big(\sum_{i,j=0}^7b_{ij} z_1^iz_2^j\Big)~.
\end{split}
\end{equation}
(The  numerator of $f_2(z)$ is omitted.)
As the holomorphic limit, we take 
$$
G_{0\overline{0}}\rightarrow 1-\theta_0 H(z_1,z_2).
$$
\subsection{$\mathbb{F}_2$-case}\label{sec:F2}
Let $\Delta$ be the polyhedron $\# 4$ in Figure \ref{fig:Hirzebruch}:
\begin{equation}\nonumber
\Delta=\text{~
the convex hull of~} \{(1,0),(0,1),(-1,0),(-2,-1)\}~.
\end{equation}
This $\Delta$ is different from previous examples 
in that there are one integral point lying on the middle of an edge.
This case has several features different from the previous cases.
\subsubsection*{{\bf $\Delta$-regularity}}
The $\Delta$-regularity condition for 
$F\in \mathbb{L}(\Delta)$
is as follows\footnote{
In the last equation, the first factor comes from 
a $1$-dimensional face and the second factor comes from 
the $2$-dimensional face.}:
\begin{equation}\nonumber
\begin{split}
&F(t_1,t_2)=a_0+a_1t_1+a_2t_2+\frac{a_3}{t_1}+\frac{a_4}{t_1^2 t_2}~,
\\
&a_1a_2a_3a_4\neq 0~,\quad
(a_3^2-4a_2a_4)\big((a_0^2-4a_1a_3)^2-64a_1^2a_2a_4\big)\neq 0~.
\end{split}
\end{equation}
\subsubsection*{{\bf $\mathcal{R}_F$ and filtrations}}
\begin{equation}\nonumber
\mathcal{R}_F\cong \mathbb{C}\,1\oplus\mathbb{C}\,t_0\oplus
            \mathbb{C}\,\frac{t_0}{t_1}\oplus\mathbb{C}\,t_0^2~.
\end{equation}
The $\mathcal{I}$-filtration is 
\begin{equation}\nonumber
\mathcal{I}_1\mathbb{C}\,t_0\oplus\mathbb{C}\,t_0^2~,\quad
\mathcal{I}_2=\mathcal{I}_3=\mathcal{I}_1\oplus \mathbb{C}\,\frac{t_0}{t_1}~,
\quad
\quad\mathcal{I}_4=\mathcal{R}_F~.
\end{equation}
The $\mathcal{E}$-filtration is
\begin{equation}\nonumber
\mathcal{E}^0=\mathbb{C}\,1~,\quad
\mathcal{E}^{-1}=\mathcal{E}^0\oplus \mathbb{C}\,t_0\oplus\mathbb{C}\,\frac{t_0}{t_1}~,
\quad
\mathcal{E}^{-2}=\mathcal{R}_F~.
\end{equation}

\subsubsection*{{\bf MHS}}
$$
H^2(\T^2,C^{\circ}_a)=\mathbb{C}\omega_0\oplus PH^1(C^{\circ}_a)~,\quad
PH^1(C^{\circ}_a)=\mathbb{C}\,\rho (t_0)\oplus
            \mathbb{C}\rho(t_0/t_1)\oplus\mathbb{C}\rho(t_0^2)~.
$$
\begin{equation}\nonumber
\mathcal{W}_1=\mathbb{C}\rho(t_0)\oplus\mathbb{C}\rho(t_0^2)~,\quad
\mathcal{W}_2=\mathcal{W}_1\oplus \mathbb{C}\rho(t_0/t_1)~,\quad
\quad\mathcal{W}_4=H^2(\T^2,\mathbb{C})~.
\end{equation}
\begin{equation}\nonumber
\mathcal{E}^0=\mathbb{C}\omega_0~,\quad
\mathcal{E}^{-1}=\mathcal{E}^0\oplus \mathbb{C}\rho(t_0)
    \oplus\mathbb{C}\rho(t_0/t_1)~,
\quad
\mathcal{E}^{-2}=H^2(\T^2,\mathbb{C})~.
\end{equation}
\subsubsection*{{\bf $A$-hypergeometric system}}
The lattice of relations $L(\Delta)$
is generated by two vectors
\begin{equation}\nonumber
l^{(1)}=(-2,1,0,1,0)~,\qquad l^{(2)}=(0,0,1,-2,1)~,
\end{equation}
and
the $A$-hypergeometric system 
is generated by the following differential operators:
\begin{equation}\nonumber 
\begin{split}
&\theta_{a_1}-\theta_{a_3}-2\theta_{a_4}~,\quad
\theta_{a_2}-\theta_{a_4}~,\quad
\theta_{a_1}+\theta_{a_2}+\theta_{a_3}+\theta_{a_4}+\theta_{a_0}~,
\\
&\partial_{a_1}\partial_{a_3}-\partial_{a_0}^2~,
\quad
\partial_{a_2}\partial_{a_4}-\partial_{a_3}^2~.
\end{split}
\end{equation}

Take the following local coordinates of $\moduli$:
$$ z_1=a^{l^{(1)}}=\frac{a_{1}a_{3}}{a_0^2}~,\qquad
z_2=a^{l^{(2)}}=\frac{a_{2}a_{4}}{a_3^2}~. 
$$
Then
we have
$\theta_0:=q_*\theta_{a_0}=-2\theta_{z_1}$.

With these coordinates, 
the $A$-hypergeometric system reduces to the following 
two differential operators of order $2$:
\begin{equation}\begin{split}\nonumber 
&\mathcal{L}_1=\theta_1(\theta_1-2\theta_2)-z_1(-2\theta_1)
(-2\theta_1-1)~,
\\
&\mathcal{L}_2=\theta_2^2-z_2(\theta_1-2\theta_2)(\theta_1-2\theta_2-1)~.
\end{split}
\end{equation}
Solutions about $z_1=0,z_2=0$ are as follows.
\begin{equation}
\begin{split}\nonumber 
&\varpi(z;0)=1~,\\
t_1:=\partial_{\rho_1}&\varpi(z;\rho)|_{\rho=0}=\log z_1+H(z_1,z_2)-G(z_2)~,
\\
t_2:=\partial_{\rho_2}&\varpi(z;\rho)|_{\rho=0}=\log z_2+2G(z_2)~,
\\
\partial_S F:=(&\partial_{\rho_1}^2+\partial_{\rho_1}\partial_{\rho_2})
 \varpi(z;\rho)~,
\end{split}
\end{equation}
where
\begin{equation}\nonumber
\begin{split}
\varpi(z;\rho)&=\sum_{n_1,n_2\geq 0}
\frac{(2\rho_1)_{2n_1}}{(\rho_1+1)_{n_1}(\rho_2+1)_{n_2}^2}
\frac{\Gamma(1+\rho_1-2\rho_2)}{\Gamma(1+\rho_1-2\rho_2+n_1-2n_2)}
z_1^{n_1+\rho_1}z_2^{n_2+\rho_2}~,
\\
H(z_1,z_2)&=2\!\!\sum_{\begin{subarray}{c}n_1,n_2\geq 0\\
n_1\geq 2n_2\end{subarray}}
\frac{(2n_1-1)!}{n_1!(n_1-2n_2)! n_2!^2}z_1^{n_1}z_2^{n_2}~,
\\
G(z_2)&=\sum_{n_2\geq 1}\frac{(2n_2-1)!}{n_2!^2}z_2^{n_2}~.
\end{split}
\end{equation}
\subsubsection*{{\bf Yukawa coupling}}
\begin{equation}\begin{split}
Y_{1,1;0}=\frac{2c}{d(z_1,z_2)}~,\quad
Y_{1,2;0}=\frac{c(1-4z_1)}{d(z_1,z_2)}~,\quad
Y_{2,2;0}=-\frac{2cz_2(1-8z_1)}{(1-4z_2)d(z_1,z_2)}~,
\end{split}
\end{equation}
and $Y_{00;0}=4Y_{11;0}$
where $d(z_1,z_2)=(1-4z_1)^2-64z_1^2z_2$
and $c\in \mathbb{C}$ is a nonzero constant.
\subsubsection*{{\bf Comparison with local A-model Yukawa coupling}}
We show that
\begin{equation}\label{eq:F2yukawa}
\mathrm{Yuk}(\partial_{t_{\alpha}},\partial_{t_{\beta}};-2\partial_{t_{1}})
\propto \partial_{t_{\alpha}}\partial_{t_{\beta}}\partial_S F~
\quad (1\leq \alpha,\beta\leq 2)~.
\end{equation}

Note that the Wronskian
 $\wronskian_{i_1\ldots i_k}(t_1,t_2,\partial_S F)$
defined as in \eqref{eq:defwronskian} is divisible by
$\theta_2 t_2$ due to the  fact that $t_2$ does not depend on $z_1$.
We define the modified Wronskian%
\footnote{
A reason to consider the modified Wronskian in the $\mathbb{F}_2$-case 
is that
the Wronskians do not satisfy the statement
corresponding to the second one
in Lemma \ref{prop:diffyukawa1}.
}
by 
\begin{equation}\nonumber
\begin{split}
\wronskian_{ i_1 \ldots i_k}'(t_1,t_2,\partial_SF)
&:=\wronskian_{i_1\ldots i_k}(t_1,t_2,\partial_S F)/\theta_2 t_2
\\
&=
\theta_1t_1 \cdot
\sum_{\alpha,\beta=1}^2
\partial_1t_{\alpha}\cdot\partial_2 t_{\beta} \cdot 
\partial_{t_{\alpha}}\partial_{t_{\beta}}\partial_S F
~~.
\end{split}
\end{equation}
As in the $\mathbb{F}_0$-case,
Lemma \ref{prop:diffyukawa1}  holds
if we replace $Y_{i_1\ldots i_k:0}$ by 
$\wronskian_{i_1\ldots i_k}'(t_1,t_2,\partial_S F)$.
Therefore $\wronskian_{ij}(t_1,t_2,\partial_S F)$
is proportional to the Yukawa coupling $Y_{ij;0}$.
Then \eqref{eq:F2yukawa} follows 
from the multi-linearity of $\mathrm{Yuk}$.
\subsubsection*{{\bf Holomorphic ambiguities}}
The multiplication constant of $Y_{00;0}$ is $c=1$ and
$$
\kappa=\frac{8z_1(1-6z_1+24z_1z_2)}{d(z_1,z_2)}~.
$$
We checked that $\tilde{C}_1^1,\tilde{C}_0^2$ 
give the correct local GW invariants of $\mathbb{F}_2$ for small degrees.
The holomorphic ambiguities are
\begin{equation}\nonumber
\begin{split}
f_1^1(z)&=-\frac{1}{12}\frac{\theta_0 d(z_1,z_2)}{d(z_1,z_2)}+\frac{1}{6}~,
\quad
f_2(z)=\frac{1}{d(z_1,z_2)^2}\Big(\sum_{i,j=0}^7b_{ij} z_1^iz_2^j\Big)~.
\end{split}
\end{equation}
(The  numerator of $f_2(z)$ is omitted.)
As the holomorphic limit, we take 
$$
G_{0\overline{0}}\rightarrow 1-\theta_0 H(z_1,z_2).
$$
\appendix
\section{Mixed Hodge structure of an Open threefold}
\label{section:Threefold}
In this section, $\Delta$ is a $2$-dimensional reflexive polyhedron.
Let $F_a\in \mathbb{L}_{\mathrm{reg}}(\Delta)$ be a $\Delta$-regular 
Laurent polynomial. 
Define $P_a\in \mathbb{C}[t_1^{\pm},t_2^{\pm},x,y]$ by
\begin{equation}\nonumber 
P_a(t_1,t_2,x,y)=F_a(t_1,t_2)+xy~.
\end{equation}
Let $Z_a^{\circ}$ be the affine
hypersurface in $\T^2\times \mathbb{C}^2$ defined by $P_a$:
\begin{equation}\begin{split}
Z_a^{\circ}&:=
\{(t_1,t_2,x,y)\in \T^2\times \mathbb{C}^2\mid 
F_a(t_1,t_2)+xy =0\}~.
\end{split}
\end{equation}
It is easy to see that the
$\Delta$-regularity of $F_a$ implies 
the smoothness of $Z_a^{\circ}$.

The goal of the appendix is to give an explicit  description of 
the MHS on $H^3(Z^{\circ}_a)$.
First we show that $H^3(Z^{\circ}_a) \cong \mathcal{R}_{F_a}$.
Next we compactify $Z^{\circ}_a$
as a hypersurface in a smooth toric variety.
Then using this compactification, we 
compute the Hodge and weight filtrations on $H^3(Z^{\circ}_a)$.
We use Batyrev's method 
for affine hypersurfaces in algebraic tori
\cite[\S6--8]{Batyrev}
with some modifications. 
\subsection{Middle cohomology $H^3(Z^{\circ}_a)$}
We have a long exact sequence
\begin{equation}\label{eq:PoincareResidueMap}
\cdots \rightarrow H^4(\T^2\times \mathbb{C}^2)\rightarrow 
H^4(\T^2\times \mathbb{C}^2\setminus Z_a^{\circ})
\stackrel{\mathrm{Res}}{\longrightarrow} 
H^3(Z_a^{\circ})\rightarrow
H^5(\T^2\times \mathbb{C}^2)\rightarrow \cdots\, .
\end{equation}
Since $H^4(\T^2\times \mathbb{C}^2)=H^5(\T^2\times \mathbb{C}^2)=0$,
the {\it Poincar\'e residue map} $\mathrm{Res}:H^4(\T^2\times \mathbb{C}^2\setminus Z_a^{\circ})\to H^3(Z_a^{\circ})$ is an isomorphism.

In the rest of this subsection, $t^m$ stands for $t_1^{m_1}t_2^{m_2}$.
By Grothendieck \cite{Grothendieck},
$H^{\bullet}(\T^2\times \mathbb{C}^2\setminus Z_a^{\circ})$ 
is isomorphic to the cohomology of
the global de Rham complex 
$(\Gamma\Omega^{\bullet}_{\T^2\times \mathbb{C}^2}(*Z_a^{\circ}),d)$
of meromorphic differential forms on $\T^2\times\mathbb{C}^2$ with poles of
arbitrary order on $Z_a^{\circ}$.
Let $R'$ be the homomorphism:
$$
R':\ring_{\Delta}\rightarrow
\Gamma\Omega^4_{\T^2\times\mathbb{C}^2}(*Z_a^{\circ})~,\quad
t_0^kt^m\mapsto 
\frac{(-1)^kk!\,t^m}{{P_a}^{k+1}}
\frac{dt_1}{t_1}\frac{dt_2}{t_2}dxdy~.
$$
\begin{proposition} \label{prop:A1}
The map $R'$ induces an isomorphism
\begin{equation}\nonumber
R': \mathcal{R}_{F_a}\overset{\cong}{\rightarrow}
H^4(\T^2\times \mathbb{C}^2\setminus Z_a^{\circ})~.
\end{equation}
\end{proposition}

\begin{corollary}\label{prop:A2}
The map $R'$ and the Poincar\'e residue map give an isomorphism
$$
\rho':  \mathcal{R}_{F_a}\overset{\cong}{\rightarrow} H^3(Z_a^{\circ})~.
$$
\end{corollary}
\begin{remark}\label{rem:dim}
For $i=0,1,2$, $\iota^* : H^i(\mathbb{T}^2\times \mathbb{C}^2)\to
H^i(Z^{\circ}_a)$ is an isomorphism where $\iota:Z^{\circ}_a \to \T^2\times \mathbb{C}^2$ 
is the inclusion. For $i\geq 4$, $H^i(Z^{\circ}_a)=0$ since $Z^{\circ}_a$ is affine. 
\end{remark}

\begin{proof}(of Proposition \ref{prop:A1}.)
We would like to compute
\begin{equation}\label{eq:H4}
\frac{\Gamma\Omega^4_{\T^2\times \mathbb{C}^2}(*Z^{\circ}_a)}
{d\Gamma\Omega^3_{\T^2\times \mathbb{C}^2}(*Z^{\circ}_a)}
\cong H^4(\T^2\times \mathbb{C}^2\setminus Z^{\circ}_a)~.
\end{equation}

Let 
\begin{equation}\nonumber
\mathbf{M}_0=\mathbb{C}[t_0,t_1^{\pm},t_2^{\pm},x,y]~,
\quad
\mathbf{M}=\mathbf{M}_0/\mathcal{D}_0'\mathbf{M}_0~,
\quad
\mathbf{L}=
\mathbb{C}[t_1^{\pm},t_2^{\pm}]~,
\end{equation}
where $\mathcal{D}'_0:\M_0\to\M_0$ is defined by
$$
\mathcal{D}'_0(t_0^k t^m x^{m_3}y^{m_4}):=
\begin{cases}(k+t_0 P_a) t_0^k t^m x^{m_3}y^{m_4}~ &(k>0)
\\
(1+t_0P_a)t_0^kt^m x^{m_3}y^{m_4}& (k=0)~
\end{cases}~~.
$$

We first rewrite the left-hand-side of \eqref{eq:H4}
using $\mathbf{M}$.
Consider the homomorphism $\Psi_0: \M_0 \to
\Omega_{\T^2\times \mathbb{C}^2}^0(*Z_a^{\circ})$ given by
\begin{equation}\nonumber
\Psi_0(t_0^kt^m x^{m_3}y^{m_4})=\begin{cases}
\frac{(-1)^{k-1}(k-1)! t^m x^{m_3}y^{m_4}}{P_a^k}~&(k\geq 1)\\
-t^m x^{m_3}y^{m_4}~&(k=0)
\end{cases}~~.
\end{equation}
Then the kernel of $\Psi_0$ is  $\mathcal{D}'_0\mathbf{M}_0$.
Therefore $\Psi_0$ induces an isomorphism $\Psi:\mathbf{M}
\stackrel{\cong}{\rightarrow} \Omega_{\T^2\times \mathbb{C}^2}^0(*Z^{\circ}_a)$.

Define the operators $\mathcal{D}'_i$ $(1\leq i\leq 4)$ acting on $\mathbf{M}$ by
\begin{equation}\nonumber
\begin{split}
\mathcal{D}'_i(t_0^kt^m x^{m_3}y^{m_4})&:=\begin{cases}
(m_i+t_0\theta_{t_i}P_a)t_0^kt^m x^{m_3}y^{m_4}&(i=1,2,~k>0)~
\\m_it_0^kt^m x^{m_3}y^{m_4}&(i=1,2~,k=0)~
\end{cases}
\\
\mathcal{D}'_3(t_0^kt^m x^{m_3}y^{m_4})&:=
\begin{cases}
(m_3+t_0\theta_{x} P_a)t_0^kt^m x^{m_3-1}y^{m_4}&(k>0)~
\\m_3 t_0^k t^m x^{m_3-1}y^{m_4}&(k=0)~
\end{cases}
\\
\mathcal{D}'_4(t_0^kt^m x^{m_3}y^{m_4})&:=
\begin{cases}
(m_4+t_0\theta_{y} P_a)t_0^kt^m x^{m_3}y^{m_4-1}&(k>0)~
\\m_4 t_0^k t^m x^{m_3}y^{m_4-1}&(k=0)~
\end{cases}
~~.
\end{split}
\end{equation}
Let $e_1,\ldots, e_4$ be the standard basis on $\mathbb{C}^4$.
For $I=\{i_1,\ldots,i_p\}\subset \{1,2,3,4\}$,  let
$e_I:=e_{i_1}\wedge\cdots\wedge e_{i_p}$.
Then we have an isomorphism
\begin{equation}\nonumber
\begin{split}
\Psi_p:&\M\otimes \wedge^p \mathbb{C}^4\stackrel{\sim}{\rightarrow} 
\Omega_{\T^2\times \mathbb{C}^2}^p(*Z^{\circ}_a)~;\quad
\sum_{i=1}^4f_i \otimes e_I\mapsto
\sum_{i=1}\Psi(f_i)\gamma(e_I)~,
\end{split}\end{equation}
where $\gamma$ is defined by
\begin{equation}\nonumber
\gamma(e_i)=\begin{cases}
\frac{dt_i}{t_i}&(i=1,2)\\
dx &(i=3)\\
dy &(i=4)
\end{cases}~~,\quad\gamma(e_I)=\gamma(e_{i_1})\wedge\cdots\wedge \gamma(e_{i_p})~.
\end{equation}
If we define  $\mathcal{D}':\mathbf{M}\otimes \wedge^3\mathbb{C}^4
\to \mathbf{M}\otimes \wedge^4 \mathbb{C}^4$ by
$$
\mathcal{D}'(f_I\otimes e_I):=\sum_{i=1}^4
\mathcal{D}'_i(f_I)\gamma(e_i\wedge e_I)~,
$$ 
we have a commutative diagram
\begin{equation}\nonumber
\begin{array}{ccc}
\M\otimes\wedge^3 \mathbb{C}^4&\stackrel{\D}{\rightarrow}&
\M\otimes \wedge^4 \mathbb{C}^4
\\
\Psi_3\downarrow&&\downarrow \Psi_4
\\
\Omega_{\T^2\times \mathbb{C}^2}^3(*Z_a^{\circ})&
\stackrel{d}{\rightarrow}
&\Omega_{\T^2\times \mathbb{C}^2}^4(*Z_a^{\circ})
  \end{array}
~.
\end{equation}
Thus we have
\begin{equation}\nonumber
\frac{\Gamma\Omega^4_{\T^2\times \mathbb{C}^2}(*Z^{\circ}_a)
}
{
d\Gamma\Omega^3_{\T^2\times \mathbb{C}^2}(*Z^{\circ}_a)}
\cong
\frac{\M\otimes \wedge^4 \mathbb{C}^4}{\D\M\otimes \wedge^3 \mathbb{C}^4}~.
\end{equation}
Then the proposition follows from the next lemma.
\begin{lemma} 
1. The homomorphism $\mathbf{L}[t_0] \to \M\otimes \wedge^4 \mathbb{C}^4$ 
given  by
$t_0^kt^m \mapsto t_0^{k+1}t^m\otimes e_1\wedge e_2\wedge e_3\wedge e_4$ 
induces an isomorphism
\begin{equation}\nonumber
 \mathbf{L}[t_0]/\sum_{i=0}^2\mathcal{D}_i\mathbf{L}[t_0]
\stackrel{\sim}{\longrightarrow}
\frac{\M\otimes \wedge^4 \mathbb{C}^4}{\D (\M\otimes\wedge^3 \mathbb{C}^4)}~.
\end{equation}
Here $\mathcal{D}_i$ are the same as those defined in \eqref{def:mathcalD}.
\\
2. The inclusion $\ring_{\Delta}\to \mathbf{L}[t_0]$ induces an isomorphism
\begin{equation}\nonumber
\mathbf{S}_{\Delta}
/\sum_{i=0}^2\mathcal{D}_i\ring_{\Delta}\stackrel{\sim}{\longrightarrow}
\mathbf{L}[t_0]/\sum_{i=0}^2\mathcal{D}_i\mathbf{L}[t_0]~.
\end{equation}
\end{lemma}
Proof of the lemma is by brute force calculation.
\end{proof}
\subsection{Compactification $Z_a$ of $Z^{\circ}_a$}\label{sec:compact}
\begin{figure}[t]
\unitlength .07cm
\begin{center}
\begin{picture}(35,35)(-15,-15)
\thinlines
\put(-15,0){\line(1,0){35}}
\put(0,-15){\line(0,1){35}}
\thicklines
\put(10,0){\line(-1,1){10}}
\put(10,0){\circle*{2}}
\put(11,1){(1,0)}
\put(10,0){\line(-2,-1){20}}
\put(0,10){\circle*{2}}
\put(1,11){(0,1)}
\put(0,10){\line(-1,-2){10}}
\put(-13,-14){(-1,-1)}
\put(-10,-10){\circle*{2}}
\put(0,0){\circle*{2}}
\end{picture}
\hspace*{1cm}
\begin{picture}(40,40)(-15,-15)
\thinlines
\put(-15,0){\line(1,0){40}}
\put(0,-15){\line(0,1){40}}
\thicklines
\put(20,-10){\line(-1,1){30}}
\put(20,-10){\circle*{2}}
\put(21,-13){(2,-1)}
\put(-10,20){\line(0,-1){30}}
\put(-10,20){\circle*{2}}
\put(-15,21){(-1,2)}
\put(-10,-10){\line(1,0){30}}
\put(-13,-14){(-1,-1)}
\put(-10,-10){\circle*{2}}
\put(0,0){\circle*{2}}
\put(10,0){\circle*{2}}
\put(0,10){\circle*{2}}
\put(-10,10){\circle*{2}}
\put(-10,0){\circle*{2}}
\put(0,-10){\circle*{2}}
\put(10,-10){\circle*{2}}
\end{picture}
\hspace*{1cm}
\begin{picture}(40,40)(-15,-15)
\thinlines
\put(-15,0){\line(1,0){40}}
\put(0,-15){\line(0,1){40}}
\put(20,-10){\line(-1,1){30}}
\put(-10,20){\line(0,-1){30}}
\put(-10,-10){\line(1,0){30}}

\thicklines
\put(0,0){\line(1,0){23}}
\put(0,0){\line(0,1){23}}
\put(0,0){\line(-1,2){15}}
\put(0,0){\line(-1,1){15}}
\put(0,0){\line(-1,0){15}}
\put(0,0){\line(-1,-1){15}}
\put(0,0){\line(0,-1){15}}
\put(0,0){\line(1,-1){15}}
\put(0,0){\line(2,-1){30}}
\end{picture}
\end{center}
\caption{An example of a reflexive polyhedron $\Delta$ (left),
its dual polytope $\Delta^*$ (middle) and the fan $\Sigma(\Delta)$ 
(right).
}\label{fig:dualpolytope}
\end{figure}
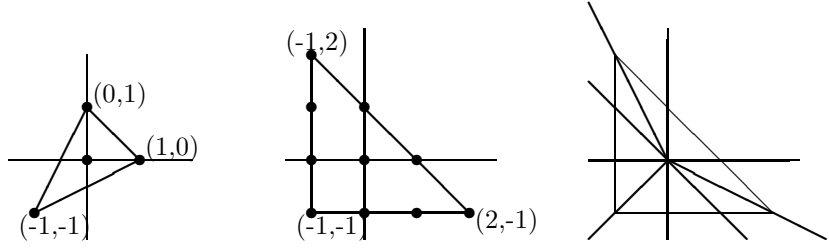
In \S \ref{sec:compact}--\S \ref{sec:deformation},
we omit the subscript $a$ from $F_a$, $P_a$, $Z_a$, $Z^{\circ}_a$, 
$C_a^{\circ}$ 
and $C_a$ for simplicity.
In \S \ref{sec:compact} and \S \ref{sec:hodge},
$t^m$ stands for the Laurent monomial
$t_1^{m_1}t_2^{m_2}t_3^{m_3}t_4^{m_4}$.

We construct a compactification of $Z^{\circ}$
as a semiample smooth hypersurface $Z$ in a $4$-dimensional
toric variety $\V$ such that the divisor 
$D=Z\setminus Z^{\circ}$ is a simple normal crossing divisor: 
\begin{equation}\nonumber
\begin{array}{ccccc}
Z^{\circ} &&\subset&& Z
\\
\cap &&&&\cap\\
\T^2\times\mathbb{C}^2&&\subset&&\V
\end{array}~~.
\end{equation}

The basic idea is to consider the following slightly modified expression for
$P=F+xy$:
\begin{equation}\label{def:tildeP}
\tilde{P}:=\frac{F(t_1,t_2)}{t_3t_4}+\frac{b_1}{t_4}+b_0~,
\quad(b_1,b_0\neq 0, F\in \mathbb{L}_{{\rm reg}}(\Delta))~.
\end{equation}
The Newton polyhedron $\tilde{\Delta}$ of $\tilde{P}$ is 
given by
\begin{equation}\label{def:tildeDelta}
\begin{split}
\tilde{\Delta}:=\big\{(m_1,m_2,&m_3,m_4)\in \mathbb{R}^4
\mid
\\&m_3\leq 0,~m_4\geq -1,~m_3-m_4\geq 0,
(m_1,m_2)\in \Delta(-m_3)\big\}~.
\end{split}
\end{equation}
Then by the general theory of the toric variety,
we obtain a singular projective 
toric variety 
$\V'=\mathrm{Proj}\,\mathbf{S}_{\tilde{\Delta}}$ such that
$H^0(\V',\mathcal{O}(1))\cong \oplus_{m\in A(\tilde{\Delta})}\mathbb{C}t^m$.
We blow up $\V'$ to obtain a smooth toric variety $\V$.
A compactification of $Z^{\circ}$
can  be obtained as a hypersurface defined by
a generic section of the pull-back of $\mathcal{O}(1)$.

Such a $\V$ can be explicitly given as follows.
First, let $v_i\in \mathbb{Z}^2$ ($1\leq i\leq r$, $r:=l(\Delta)-1$) 
be the primitive vectors
lying on faces of the dual polyhedron $\Delta^*$ of $\Delta$.
Let $\Sigma(\Delta)$ 
be the $2$-dimensional complete fan
spanned by $v_1,\ldots, v_r$ (see Figure \ref{fig:dualpolytope})
and  let  $\mathbb{P}_{\Sigma(\Delta)}$ be the corresponding smooth toric surface. 
Then $\mathbb{P}_{\Sigma(\Delta)}$ 
is a resolution of the singular toric surface
$\mathrm{Proj}\, {\mathbf{S}_{\Delta}}$,
and $C^{\circ}$ can be compactified smoothly to $C$
in $\mathbb{P}_{\Sigma(\Delta)}$.
Next we set
\begin{equation}\nonumber
\begin{split}
{\tilde{ v}}_i=\begin{pmatrix}v_i\\-1\\0\end{pmatrix}~~(1\leq i\leq r),
\quad
{u}_1=\begin{pmatrix}\vec{0}\\1\\0\end{pmatrix}~,\quad
{u}_2=\begin{pmatrix}\vec{0}\\0\\1\end{pmatrix}~,\quad
{u}_3=\begin{pmatrix}\vec{0}\\-1\\0\end{pmatrix}~,\quad
{u}_4=\begin{pmatrix}\vec{0}\\1\\-1\end{pmatrix}~.
\end{split}
\end{equation}
Then the $1$-cones of the fan $\Sigma_{\V}$ are given by
$$ 
\nu_i=\mathbb{R}_{\geq 0}{\tilde{ v}}_i ~~(1\leq i\leq r)~,\quad
\mu_i=\mathbb{R}_{\geq 0}{u}_j~~ (1\leq j\leq 4)~,
$$
and the $4$-cones of $\Sigma_{\V}$ are given by
\begin{equation}\nonumber
[i,i+1;j,j+1]:=\mathbb{R}_{\geq 0}{\tilde{ v}}_i+
\mathbb{R}_{\geq 0}{\tilde{v}}_{i+1}+
\mathbb{R}_{\geq 0}{u}_j+
\mathbb{R}_{\geq 0}{u}_{j+1}
\quad
(1\leq i\leq r,~1\leq j\leq 4).
\end{equation}
(For the sake of convenience, 
we set ${\tilde{v}}_{r+1}:={\tilde{v}}_1, { u}_5:={u}_1$
and $\nu_{r+1}:=\nu_1, \mu_5:=\mu_1$.)
The $3$-cones and the $2$-cones are faces of the above $4$-cones.
The toric variety $\V$ associated to the fan $\Sigma_{\mathbf{V}}$
is a
bundle over the toric surface $\mathbb{P}_{\Sigma(\Delta)}$ whose fiber is 
the Hirzebruch surface $\mathbb{F}_1$.

Let $\mathbb{D}_i$ ($1\leq i\leq r$) and $\mathbb{E}_j$ ($1\leq j\leq 4$)
be the toric divisors of $\V$ corresponding to
the $1$-cones $\nu_i$ and $\mu_j$ respectively.
By a standard computation in the theory of toric varieties (see e.g. \cite{Oda}), 
we have

\begin{lemma} 
$$
H^0(\V,\mathcal{O}(\mathbb{E}_1+\mathbb{E}_2))=
\bigoplus_{m\in A(\tilde{\Delta})} \mathbb{C}\,t^m~.
$$
\end{lemma}

Therefore $\tilde{P}$ in eq.\eqref{def:tildeP} is a generic section of 
the line bundle corresponding to the divisor $\mathbb{E}_1+\mathbb{E}_2$.
We define $Z$ to be the hypersurface in $\V$ defined by $\tilde{P}$.
We show that if we assume $F$ is $\Delta$-regular and $b_0b_1\neq 0$, 
then 
(1) $Z^{\circ}\subset Z$; (2) $Z$ is smooth;
(3) $D=Z/Z^{\circ}$ is a (simple) normal crossing divisor.
(1) can be shown as follows.
Let $U_{\mu_1,\mu_2}=\mathrm{Spec}\,[t_1^{\pm 1},t_2^{\pm 2},t_3,t_4]\subset \V$ be the open set
corresponding to the $2$-cone spanned by $\mu_1,\mu_2$.
It is isomorphic to $\T^2\times \mathbb{C}^2$.
The defining equation of $Z$ on $U_{\mu_1,\mu_2}$ is 
$
F(t_1,t_2)+b_1 t_3+b_0t_3t_4
$
and this is equal to $P$ if
we identify $x=t_3, y=b_1+b_0t_4$. 
We can prove (2) and (3) by looking at 
the defining equation $P_{\sigma}$ of $Z$ 
on the open subset $U_{\sigma}\subset
\V$ corresponding to each 4-cone $\sigma$.

We end this subsection by listing the Hodge numbers of 
$Z$. 
\begin{lemma}\label{lem:hodgenumber}
The Hodge numbers $h^{p,q}(Z)=\dim H^{p,q}(Z)$ 
are 
\begin{equation}\nonumber
\begin{array}{r|rccc}
   &p=0&1&2&3\\\hline
q=0&1&0&0&0\\
1& 0&l(\Delta)-1&1&0\\
2&  0&1&l(\Delta)-1&0\\ 
3&  0&0&0&1
\end{array}~~~~.
\end{equation}
\end{lemma}
\begin{proof} 
By the formula on cohomology of semiample divisors on 
a toric variety due to Mavlyutov \cite[Cor.2.7]{Mavlyutov}%
\footnote{
For a semiample toric divisor $X$ in a 
$d$-dimensional complete simplicial toric variety $\mathbb{P}_{\Sigma}$,
Mavlyutov's formula is:
$$\dim H^k(\mathbb{P}_{\Sigma},\Omega_{\mathbb{P}_{\Sigma}}^l(X))
=\sum_{\delta}l^*(\delta)
\begin{pmatrix}\dim \delta\\l-k \end{pmatrix}
\cdot
\sum_{j=0}^k\begin{pmatrix}
d-\dim \delta -j\\k-j
\end{pmatrix}
(-1)^{k-j}
\#\Sigma_{\sigma_{\delta}}(j)~.
$$
The sum is over all faces $\delta$ of the polytope $\Delta_X$
associated to $X$, 
$l^*(\delta)$ is the number of interior integral points in the face $\delta$,
$\sigma_{\delta}\in \Sigma_X$
is a cone corresponding to the face $\delta$ in
the fan $\Sigma_X$
(which
is the fan such that 
there is a morphism $s:\Sigma\to \Sigma_X$;
$X$ is a pull-back of an ample divisor
by the induced morphism of toric varieties 
$\mathbb{P}_{\Sigma}\to \mathbb{P}_{\Sigma_X}$),
and
$\#\Sigma_{\sigma_{\delta}}(j)$ is the number of $j$-cones
in $\Sigma_{\sigma_{\delta}}=\{s(\tau)\in \Sigma: \tau \in \sigma_{\delta}\}$.
In the case of the threefold $Z\subset \V$,
$\Sigma_{Z}$ is generated by $u_2,u_4$ and 1-cones $\tilde{v}_i$ 
such that $v_i$ are vertices of $\Delta^*$.},
we can explicitly 
compute the dimensions of $H^q(\V,\Omega^p(\mathbb{E}_1+\mathbb{E}_2))$
and $H^q(\V,\Omega^p(2\mathbb{E}_1+2\mathbb{E}_2))$.
Then we obtain $\dim H^q(Z, \Omega^p)$ by 
exact sequences (as in the proof of the Lefshetz hyperplane theorem \cite[p.156]{GriffithsHarris}).
\end{proof}
\subsection{The Hodge filtration}\label{sec:hodge}
Let $\mathbb{D}=\sum_{i=1}^r \mathbb{D}_i+\mathbb{E}_3+\mathbb{E}_4$.
Note that $\mathbb{D}=\mathbf{V}\setminus \T^2\times \mathbb{C}^2$
and $D=Z/Z^{\circ}=Z\cap \mathbb{D}$.
\begin{proposition}\label{prop:hodge1}
For $p=1,2,3,4$, the residue mapping
\begin{equation}\nonumber
H^{4-p}\big(\V,\Omega_{\V}^p(\log(Z+\divisor))\big)
 \stackrel{\res_{Z}}{\rightarrow}
H^{4-p}\big(Z, \Omega_{Z}^{p-1}(\log D)\big)
\end{equation}
is an isomorphism.
\end{proposition}
\begin{proof}
Consider the exact sequence
\begin{equation}\nonumber
0\rightarrow \Omega_{\V}^p(\log\divisor)
\rightarrow\Omega_{\V}^p(\log(Z+\divisor))
\stackrel{\res_{Z}}{\rightarrow}
\Omega^{p-1}_{Z}(\log D)\rightarrow 0~,
\end{equation}
and take the cohomology.
The vanishings $H^{4-p}(\V,\Omega^p_{\V}(\log\divisor))=H^{5-p}(\V,\Omega^p_{\V}(\log\divisor))=0$ imply the proposition.
\end{proof}

We use the notation
$\Omega^p_{\V,\divisor}(k):=\Omega_{\V}^p(\log\divisor)
 \otimes \mathcal{O}(kZ) $
for integers $k,p\geq 0$.
\begin{proposition}\label{prop:hodge2}
\begin{equation}\nonumber
\begin{split}
H^{4-p}\big(\V,\Omega^p(\log(Z+\divisor))\big)&\cong
\frac{H^0(\V,\Omega^4_{\V,\divisor}(5-p))}
{H^0(\V,\Omega^4_{\V,\divisor}(4-p))+
dH^0(\V,\Omega^{3}_{\V,\divisor}(4-p))
}~ \quad(p=1,2,3),
\\
H^0\big(\V,\Omega^4(\log(Z+\divisor))\big) &\cong
H^0(\V,\Omega^4_{\V,\divisor}(1))~.
\end{split}
\end{equation}
\end{proposition}
\begin{proof}
The proposition follows from the exact sequence
\begin{equation}\nonumber
0\to \Omega_{\V}^p(\log(Z+\divisor))
 \,\hookrightarrow\,
     \Omega^p_{\V,\divisor}(1)
\stackrel{d}{\rightarrow}
  \frac{\Omega^{p+1}_{\V,\divisor}(2)}{\Omega^{p+1}_{\V,\divisor}(1)}
\stackrel{d}{\rightarrow}
\cdots
\stackrel{d}{\rightarrow}
  \frac{\Omega^4_{\V,\divisor}(5-p)}{\Omega^4_{\V,\divisor}(4-p)}
\rightarrow 0~,
\end{equation}
and Lemma \ref{prop:cohom3} below.
\end{proof}

\begin{lemma}\label{prop:cohom3} 
Let $k$ be a nonnegative integer, $p=1,2,3,4$.
\begin{equation}\nonumber
\begin{split}
1.\quad& H^q(\V,\Omega^p(\log (\divisor+\mathbb{E}_1+\mathbb{E}_2))
\otimes \mathcal{O}(kZ))=0 \quad(q>0)~,
\\
2.\quad&H^q(\V,\Omega^p(\log (\divisor+\mathbb{E}_i))\otimes \mathcal{O}(kZ))=0\quad(q>0,i=1,2)~,
\\
3.\quad &H^q(\V,\Omega^p_{\V,\divisor}(k))=0
\quad (q>0)~.
\end{split}
\end{equation}.
\end{lemma}

\begin{proof}
1. Let $\mathbb{D}_{\T}:=\divisor +\mathbb{E}_1+\mathbb{E}_2$. 
Note that this is the sum of all toric divisors in $\V$.
It is well known that
$\Omega_{\V}^p(\log\mathbb{D}_{\T})\cong 
         \mathcal{O}_{\V}\otimes \wedge^p M$ 
where $M$ is the dual lattice of $N\cong \mathbb{Z}^4$ (cf. \cite{Oda}).
On the other hand, 
since $\mathbb{E}_1+\mathbb{E}_2$ is semiample,
we have 
$H^q(\V,\mathcal{O}_{\V}(kZ))=0$ for $q>0$ \cite{Mavlyutov}.
Therefore
\begin{equation}\nonumber
H^q(\V,\Omega_{\V}^p(\log \mathbb{D}_{\T})\otimes \mathcal{O}_{\V}(kZ))
\cong H^q(\V,\mathcal{O}_{\V}(kZ))\otimes \wedge^p M~=0 \quad(q>0)~.
\end{equation}
2. 
As above, the following vanishing holds:
$$
H^q(\mathbb{E}_2,\Omega_{\mathbb{E}_2}^{p-1}\big(\log(\divisor+\mathbb{E}_1)\big)
\otimes\mathcal{O}_{\mathbb{E}_2}(kZ))
\cong H^q(\mathbb{E}_2,\mathcal{O}_{\mathbb{E}_2}(kZ))
\otimes \wedge^{p-1}\mathbb{Z}^3
=0\quad(q>0).
$$
Moreover, the map
$$
H^0(\V, \Omega_{\V}^p (\log\mathbb{D}_{\T})
 \otimes\mathcal{O}_{\V}(kZ)) \stackrel{\res_{\mathbb{E}_2}}{\rightarrow}
H^0(\mathbb{E}_2,
 \Omega_{\mathbb{E}_2}^{p-1}\big(\log(\divisor+\mathbb{E}_1)\big)
\otimes\mathcal{O}_{\mathbb{E}_2}(kZ)
)
$$
is surjective.
Taking the exact sequence of cohomology of
the exact sequence:
\begin{equation}\nonumber
\begin{split}
0\to 
\Omega_{\V}^p\big(\log (\divisor+\mathbb{E}_{1})\big)
&\otimes 
  \mathcal{O}_{\V}(kZ)
\rightarrow
\Omega_{\V}^p (\log\mathbb{D}_{\T})
 \otimes\mathcal{O}_{\V}(kZ)
\\&\stackrel{\res_{\mathbb{E}_2}}{\rightarrow}
\Omega_{\mathbb{E}_2}^{p-1}\big(\log(\divisor+\mathbb{E}_1)\big)
\otimes\mathcal{O}_{\mathbb{E}_2}(kZ)\to 0~,
\end{split}
\end{equation}
we obtain $$
H^q(\V,\Omega^p(\log (\divisor+\mathbb{E}_1))\otimes \mathcal{O}(kZ))=0
\quad (q>0)~.$$
The proof for $\mathbb{E}_2$ is similar.
\\
3. 
As above, we can show the vanishing
$$
H^q(\mathbb{E}_1\cap \mathbb{E}_2,
\Omega_{\mathbb{E}_1\cap \mathbb{E}_2}^{p-2}(\log\divisor)
\otimes\mathcal{O}(kZ) )
\cong H^q(\mathbb{E}_1\cap\mathbb{E}_2,
 \mathcal{O}_{\mathbb{E}_1\cap \mathbb{E}_2}(kZ))
\otimes \wedge^{p-2}\mathbb{Z}^2
=0\quad(q>0),
$$
and the surjectivity of the map
$$
\bigoplus_{i=1,2}
H^0(\V,\Omega_{\V}^p\big(\log (\divisor+\mathbb{E}_i)\big)
\otimes \mathcal{O}(kZ))
\stackrel{\res_{\mathbb{E}_1,\mathbb{E}_2}}{\rightarrow}
H^0(\mathbb{E}_1\cap \mathbb{E}_2,
\Omega_{\mathbb{E}_1\cap \mathbb{E}_2}^{p-2}(\log\divisor)
\otimes\mathcal{O}(kZ) ).
$$
Consider the exact sequence
\begin{equation}\nonumber
\begin{split}
0&\to \Omega_{\V}^p(\log \divisor)\otimes \mathcal{O}(kZ)
\to
\bigoplus_{i=1,2}\Omega_{\V}^p\big(\log (\divisor+\mathbb{E}_i)\big)\otimes \mathcal{O}(kZ)
\\
&\rightarrow
\Omega_{\V}^p(\log \mathbb{D}_{\T})
 \otimes\mathcal{O}(kZ)
\stackrel{\res_{\mathbb{E}_1,\mathbb{E}_2}}{\longrightarrow}
\Omega_{\mathbb{E}_1\cap \mathbb{E}_2}^{p-2}(\log\divisor)
\otimes\mathcal{O}(kZ)\rightarrow 0~.
\end{split}
\end{equation}
Taking the cohomology,
we obtain the statement 3. 
\end{proof}

Let $\tilde{\Delta}(k)$ be the polyhedron defined 
by applying \eqref{def:ktriangle}
to $\tilde{\Delta}$ defined in \eqref{def:tildeDelta}
and let 
$\tilde{\Delta}[k]$ be the following $4$-dimensional polyhedron:
\begin{equation}\nonumber
\begin{split}
\tilde{\Delta}[k]:=
\big\{&
(m_1,m_2,m_3,m_4)\in\mathbb{R}^4 \mid
\\&(m_1,m_2)\in \Delta(-m_3-1),~
-k+1\leq m_3\leq 0,~m_4\geq -k,~m_3-m_4\geq 0~\big\}~.
\end{split}
\end{equation}

\begin{proposition}\label{prop:hodge3}
Let $k$ be a positive integer $k\geq 1$. 
\begin{equation}\nonumber
\begin{split}
1.\quad&
H^0(\V,\Omega^4_{\V,\divisor}(k))=
\bigoplus_{m\in \tilde{\Delta}(k-1)\cap\mathbb{Z}^4}\mathbb{C}\,
\frac{t^m }{\tilde{P}^k}
\frac{dt_1}{t_1}\frac{dt_2}{t_2}
  \frac{dt_3}{t_3}\frac{dt_4}{t_4}~.
\\
2.\quad&
H^0(\V,\Omega^3_{\V,\divisor}(k))=
\bigoplus_{m\in \tilde{\Delta}[k]\cap \mathbb{Z}^4}
\mathbb{C}\,\frac{t^m}{\tilde{P}^k}\frac{dt_1}{t_1}\frac{dt_2}{t_2}\frac{dt_3}{t_3}\,
\\&
\oplus
\bigoplus_{m\in \tilde{\Delta}(k-1)\cap\mathbb{Z}^4}\Big[
\mathbb{C}\,\frac{t^m}{\tilde{P}^k}\frac{dt_1}{t_1}\frac{dt_2}{t_2}\frac{dt_4}{t_4}
\oplus 
\mathbb{C}\,\frac{t^m}{\tilde{P}^k}\frac{dt_1}{t_1}\frac{dt_3}{t_3}\frac{dt_4}{t_4}
\,\oplus
\mathbb{C}\,\frac{t^m}{\tilde{P}^k}\frac{dt_2}{t_2}\frac{dt_3}{t_3}\frac{dt_4}{t_4}
\Big]~.
\\
3.\quad&
\frac{H^0(\V,\Omega^4_{\V,\divisor}(k+1))}
{H^0(\V,\Omega^n_{\V,\divisor}(k))+
dH^0(\V,\Omega^{3}_{\V,\divisor}(k))}
\cong {R}^{k}_F~,
\end{split}
\end{equation}
where the isomorphism is induced from the
map
\begin{equation}\nonumber
\ring_{\Delta}^{k}
\to H^0(\V,\Omega^4_{\V,\divisor}(k+1))~,\quad
t_0^k t_1^{m_1}t_2^{m_2}\mapsto
\frac{(-1)^k k! t_1^{m_1}t_2^{m_2}}{\tilde{P}^{k+1}(t_3t_4)^k}
\frac{dt_1}{t_1}\frac{dt_2}{t_2}
\frac{dt_3}{t_3}\frac{dt_4}{t_4}~.
\end{equation}
\end{proposition}
\begin{proof}
The statements 1 and 2 can be shown by 
calculation of the Cech cohomology associated to 
the open cover given by the toric fan. 
The third statement follows from the first and the second. 
\end{proof}

\begin{theorem} \label{prop:Hodge}
The Hodge filtration on $H^3(Z^{\circ})$ satisfies 
\begin{equation}\begin{split}\nonumber
&\mathrm{Gr}_{\hodge}^{p}H^3(Z^{\circ})\cong R_F^{3-p}\, ,
\quad(0\leq p\leq 3)~.
\end{split}
\end{equation}
\end{theorem}
\begin{proof}
As is well known, there are canonical isomorphisms
${\rm Gr}_{\hodge}^p H^3(Z^{\circ})\cong H^{3-p}(Z,\Omega^{p}_Z(\log D))$.
The theorem follows from Propositions \ref{prop:hodge1}, \ref{prop:hodge2}
and \ref{prop:hodge3}.
\end{proof}
\subsection{The weight filtrations}
%
\begin{proposition}\label{prop:weight3}
\begin{equation}\begin{split}\nonumber
&\dim {\rm Gr}^\weight_6= 1 ~,\quad
\dim {\rm Gr}^{\weight}_5=0~,
\\ &
\dim {\rm Gr}^{\weight}_4=l(\Delta)-4~,\quad
\dim {\rm Gr}^{\weight}_3=2~.
\end{split}\end{equation}
\end{proposition}
\begin{proof} 
The divisor $D=Z\setminus Z^{\circ}$ consists of $r+2$ components.
Define 
$D^{(k)}$ to be the disjoint union of intersections of $k$ components for $k=1,2,3$ and $D^{(0)}:=Z$.
Consider the spectral sequence ${}_\weight E$ of
the hypercohomology $\mathbb{H}^k(Z,\Omega^{\bullet}_{Z}(\log D))$ 
associated to the decreasing weight filtration $\weight^{-l}:=\weight_l$.
This spectral sequence degenerates at ${}_{\weight}E_2$.
We have 
$
{}_{\weight}E_1^{p,q}\cong H^{2p+q}(D^{(-p)},\mathbb{C})
$
and the differential $d_1:H^{2p+q}(D^{(-p)},\mathbb{C})\to H^{2p+q+2}(D^{(-p-1)},\mathbb{C})$ is given by the Gysin morphism
(see e.g. \cite[Corollary 8.33, Proposition 8.34]{Voisin}).
Computing the cohomology of $d_1$, we obtain the following result.
\begin{equation}\nonumber
\mathrm{dim}\,{}_{\weight}E_2^{p,q}
=\hspace*{1cm}
\begin{array}{r|rccc}
&p=0&-1&-2&-3\\\hline
q=0&1&0&0&0\\
1  &0&0&0&0\\
2  &0&2&0&0\\
3  &2&0&0&0\\
4  &0&l(\Delta)-4&1&0\\
5  &0&0&0&0\\
6  &0&0&0&1
\end{array} ~~~~~.
\end{equation}
The dimensions of 
the graded quotients $
\text{Gr}_{-p+3}^{\weight}H^3(Z^{\circ})\cong 
{}_{\weight}E_2^{p,3-p}$ can be read from this table.
\end{proof}

\begin{proposition}
The weight filtration on $H^4(\T^2\times \mathbb{C}^2\setminus Z^{\circ})$
is 
\begin{equation}\begin{split}\nonumber
&\weight_8 H^4(\T^2\times \mathbb{C}^2\setminus Z^{\circ})=\mathcal{I}_4=
\mathcal{R}_F~, 
\\
&\weight_7H^4(\T^2\times \mathbb{C}^2\setminus Z^{\circ})=
\weight_6H^4(\T^2\times \mathbb{C}^2\setminus Z^{\circ})=\mathcal{I}_3~,
\\&
\weight_5 H^4(\T^2\times \mathbb{C}^2\setminus Z^{\circ})=\mathcal{I}_1~.
\end{split}
\end{equation}
\end{proposition}

\begin{proof}
We consider  
three filtrations $\mathcal{V},\mathcal{V}',\mathcal{V}''$ 
on $H^4(\T^2\times \mathbb{C}^2\setminus Z^{\circ})$
and compare them.
First  we define
\begin{equation}\nonumber
\mathcal{V}_k\Omega_{\V}^i(\log (Z+\mathbb{D})):=
\Omega_{\V}^{i-k}\wedge 
  \Omega_{\V}^k(\log(Z+\mathbb{D}))~\quad
(0\leq k\leq 4).
\end{equation}
This induces the weight filtration 
$\mathcal{V}_k H^i(\T^2\times \mathbb{C}^2\setminus Z^{\circ})
=\mathcal{W}_{k+4}H^i (\T^2\times \mathbb{C}^2\setminus Z^{\circ})$.
We have already computed the dimension of graded quotients
in Proposition \ref{prop:weight3}.

Second, let $U:=\V\setminus(Z\cup \mathbb{E}_3\cup\mathbb{E}_4)$.
We define
\begin{equation}\nonumber
\mathcal{V}'_k\Omega_U^i(\log (U\cap \mathbb{D}))
:= \Omega_U^{i-k}\wedge \Omega^k_U(\log(U\cap\mathbb{D}))~\quad(k=0,1,2).
\end{equation}
This induces another filtration $\mathcal{V}'$ on 
$H^4(\mathbb{T}^2\times\mathbb{C}^2\setminus Z^{\circ})$.
As in \cite[\S8]{Batyrev},
this is given by the $\mathcal{I}$-filtration on $\mathcal{R}_F$:
\begin{equation}\nonumber
\mathcal{V}_0'
H^4(\mathbb{T}^2\times\mathbb{C}^2\setminus Z^{\circ})
\cong \mathcal{I}_1~,\quad
\mathcal{V}_1'
H^4(\mathbb{T}^2\times\mathbb{C}^2\setminus Z^{\circ})
\cong\mathcal{I}_3~,\quad
\mathcal{V}_2'
H^4(\mathbb{T}^2\times\mathbb{C}^2\setminus Z^{\circ})
\cong\mathcal{I}_4~.
\end{equation}
Third, let $j:U\to \V$ be the inclusion and define
\begin{equation}\nonumber
\mathcal{V}''_k\Omega_{\V}^i(\log(Z+\mathbb{D}))
:=
\Omega_{\V}^{i-k}\wedge \Omega_{\V,\mathbb{D}}^k+
(\mathcal{V}_{k-1}'j_*\Omega_U^i(\log(U\cap \mathbb{D})))\cap \Omega_{\V}^i(\log (Z+\mathbb{D}))~.
\end{equation}
Since $H^4(\T^2\times \mathbb{C}^2)=0$,
the first term does not contribute to $H^4(\T^2\times \mathbb{C}^2\setminus Z^{\circ})$.
So the induced filtration is related to $\mathcal{V}'$ by 
\begin{equation}\nonumber
\begin{split}
&\mathcal{V}''_kH^4(\T^2\times \mathbb{C}^2\setminus Z^{\circ})=
\mathcal{V}'_{k-1}H^4(\T^2\times \mathbb{C}^2\setminus Z^{\circ})~
\quad(k=1,2,3)~.
\end{split}
\end{equation}
Moreover it holds that 
\begin{equation}\nonumber
\begin{split}
&\mathcal{V}_k\Omega_{\V}^i(\log(Z+\mathbb{D}))
\subset
\mathcal{V}_k''\Omega_{\V}^i(\log(Z+\mathbb{D}))~\quad(k=1,2)~,
\\
&\mathcal{V}_4\Omega_{\V}^i(\log(Z+\mathbb{D}))
\subset
\mathcal{V}_3''\Omega_{\V}^i(\log(Z+\mathbb{D}))~.
\end{split}
\end{equation}
Therefore  we have
$$
\mathcal{V}_kH^4(\T^2\times \mathbb{C}^2\setminus Z^{\circ})
\subset
\mathcal{V}_{k-1}'H^4(\T^2\times \mathbb{C}^2\setminus Z^{\circ})~~(k=1,2)~,
\quad
\mathcal{V}_4H^4(\T^2\times \mathbb{C}^2\setminus Z^{\circ})
\subset
\mathcal{V}_2'H^4(\T^2\times \mathbb{C}^2\setminus Z^{\circ})~.
$$
By the dimension consideration, we see that the proposition holds.
\end{proof}

Since taking the residue map 
$H^4(\T^2\times \mathbb{C}^2\setminus Z^{\circ})
\to H^3(Z^{\circ})$ decreases the weight by $2$, we obtain

\begin{theorem}\label{prop:weight2}
The weight filtration on 
$H^3(Z^{\circ})$ is as follows.
\begin{equation}\begin{split}\nonumber
&\weight_6H^3(Z^{\circ})\cong\mathcal{R}_F~,
\\
&\weight_5H^3(Z^{\circ})=
\weight_4 H^3(Z^{\circ})\cong\mathcal{I}_3~,
\\
&
\weight_3H^3(Z^{\circ})\cong\mathcal{I}_1~.
\end{split}
\end{equation}
\end{theorem}
\subsection{Deformation and Obstruction}\label{sec:deformation}
By Kawamata's result \cite{Kawamata},
$H^1\big(Z,T_{Z}(-\log D))\big)$
and $H^2\big(Z,T_{Z}(-\log D))\big)$
are the set of infinitesimal logarithmic deformations
and the set of obstructions respectively.

Let $\threeform$ be the following global section of 
$K_{Z}(D)=\Omega^3_Z(\log D)$ :
$$
\threeform=
\res_{Z} \frac{1}{P}\frac{dt_1}{t_1}\frac{dt_2}{t_2}
{dx}{dy}~.
$$
\begin{proposition}
\begin{equation}\begin{split}\nonumber
&H^1(Z,T_{Z}(-\log D))\stackrel{\sim}{\rightarrow}
 H^1(Z,\Omega^2_{Z}(\log D)) \cong R_F^1~,\\
&H^2(Z,T_{Z}(-\log D))\stackrel{\sim}{\rightarrow}
 H^2(Z,\Omega^2_{Z}(\log D))=0~,
\end{split}
\end{equation}
where the isomorphisms are given by the contraction with 
the three form $\threeform$.
\end{proposition}

\begin{proof}
The contraction with $\omega$ is an isomorphism
since  $K_{Z}(D)\cong \mathcal{O}_{Z}$
and $T_{Z}(-\log D)$ and $\Omega^2_{Z}(\log D)$ are locally free.
For the rest, see Theorem \ref{prop:Hodge} and Remark \ref{rem:dim}.
\end{proof}
\subsection{Variation of Mixed Hodge Structures}
\label{sec:A6}
Varying the parameter $a\in \mathbb{L}_{{\rm reg}}(\Delta)$,
we obtain a family of threefolds $Z_a^{\circ}$:
$$
p': \mathcal{Z}'\rightarrow \mathbb{L}_{{\rm reg}}(\Delta)~.
$$
We have
$$
R^3p'_*\mathbb{Z}\otimes
\mathcal{O}_{\mathbb{L}_{{\rm reg}}(\Delta)}
\cong \mathcal{R}_{F}[a]
\otimes \mathcal{O}_{\mathbb{L}_{{\rm reg}}(\Delta)}
~.
$$
The Gauss--Manin connection on $\nabla_{a_m}$ on 
$R^3p'_*\mathbb{Z}\otimes
\mathcal{O}_{\mathbb{L}_{{\rm reg}}(\Delta)}$ 
corresponds to the derivation $\mathcal{D}_{a_m}$
since it corresponds
to the differentiation by $a_m$
on $\Gamma\Omega^4_{\T^2\times \mathbb{C}^2}(*Z_a^{\circ})$.

Let $\omega_a$ be the relative holomorphic three form
on $\mathcal{Z}'$
such that $$
\omega_a
=
\mathrm{Res}\frac{1}{F_a+xy}
\frac{dt_1}{t_1}\frac{dt_2}{t_2}{dx}{dy}~.
$$
By Proposition \ref{prop:Ahyper}, we obtain
\begin{corollary}
1. $H^3(Z^{\circ}_a)$ is spanned by $\omega_{a}$, 
$\nabla_{\partial_{a_m}}\omega_{a}$ 
and $\nabla_{\partial_{a_m}}\nabla_{\partial_{a_n}}\omega_{a}$
$(m, n\in A(\Delta))$.\\
2. 
$\threeform_a$ satisfies the $A$-hypergeometric system \eqref{Ahyper}
with $\partial_{a_i}$ replaced by $\GM_{\partial_{a_i}}$.
\\
3. Period integrals of $\omega_a$
satisfies the $A$-hypergeometric system \eqref{Ahyper}.
Conversely, a solution of the $A$-hypergeometric system
\eqref{Ahyper} is a period integral.
\end{corollary}


\end{document}